\documentclass{amsart} 
\usepackage{amsmath,amsfonts,amssymb}
\usepackage{multirow}

\usepackage{float}
\usepackage{amssymb, amsthm, amsmath, url}
\usepackage[mathscr]{eucal}
\usepackage{enumerate}
\usepackage{bbm} 

\usepackage{pgf}
\usepackage{tikz}
\usetikzlibrary{positioning}
\usetikzlibrary{arrows,shapes,calc,backgrounds,fit}
\tikzset{%
 unshaded/.style={draw, shape=circle, fill=white, inner sep=1.5pt},
 shaded/.style={draw, shape=circle, fill=black, inner sep=1.5pt},
 invisible/.style={shape=circle, inner sep=1.5pt},
 label/.style={shape=rectangle, inner xsep=6pt, inner ysep=7pt},
 auto,
 curvy/.style={semithick, ->, shorten >=3pt, shorten <=3pt, >=latex, looseness=1.1, bend angle=30},
 straight/.style={semithick, ->, shorten >=3pt,shorten <=3pt, >=latex},
 loopy/.style={semithick, ->, shorten >=3pt, shorten <=3pt, >=latex, min distance=15pt},
 order/.style={thin},
 bigloop/.style={semithick, -, min distance=65pt}}
 \pgfdeclarelayer{background}
 \pgfdeclarelayer{foreground}
 \pgfsetlayers{background,main,foreground}

\numberwithin{equation}{section}

\theoremstyle{plain}
\newtheorem{theorem}{Theorem}[section]
\newtheorem{lemma}[theorem]{Lemma}
\newtheorem{corollary}[theorem]{Corollary} 
\newtheorem{lemma:trans}[theorem]{Transfer Lemma}

\theoremstyle{definition}
\newtheorem{example}[theorem]{Example}

\newtheorem{definition}[theorem]{Definition}

\newtheorem{facts}[theorem]{Facts}
\newtheorem{remark}[theorem]{Remark}

\newcommand{\A}{{\mathbf A}}
\newcommand{\Adown}{\A^{\flat}}
\newcommand{\B}{{\mathbf B}}
\newcommand{\C}{{\mathbf C}}
\newcommand{\Ok}{{\mathbf O}}

\newcommand{\TwB}{{\mathbf 2}}

\newcommand{\CT}{{\mathbb C}}

\newcommand{\TwT}{{\mathbbm 2}}

\newcommand{\X}{{\mathbb X}}
\newcommand{\Y}{{\mathbb Y}}
\newcommand{\Z}{{\mathbb Z}}

\newcommand{\T}{{\mathscr T}}

\newcommand{\cat}[1]{\boldsymbol{\mathscr{#1}}}

\newcommand{\CA}{\cat A}  
\newcommand{\CB}{\cat B}  
\newcommand{\CCF}{\cat C\!_F}  

\newcommand{\CX}{\cat X}
\newcommand{\CXF}{\cat X\!_F} 
\newcommand{\CY}{\cat Y}
\newcommand{\CD}{\cat D}
\newcommand{\CP}{\cat P}
\newcommand{\CV}{\cat V}

\DeclareMathOperator{\ISP}{\mathsf{ISP}}

\newcommand{\id}[1]{\operatorname{id}_{#1}}
\newcommand{\Con}[1]{\operatorname{Con}(#1)}
\newcommand{\Sub}[1]{\operatorname{Sub}(#1)}  
\newcommand{\Var}[1]{\operatorname{Var}(#1)}
\newcommand{\dotcup}{\mathbin{\dot{\cup}}}

\newcommand{\comp}{{\setminus}}

\renewcommand{\le}{\leqslant}
\renewcommand{\ge}{\geqslant}
\newcommand{\nle}{\nleqslant}

\renewcommand{\emptyset}{\varnothing}
\renewcommand{\phi}{\varphi}
\newcommand{\sg}[2]{\operatorname{sg}_{#1}(#2)}

\begin{document}
\title[Restricted Priestley dualities]{Restricted Priestley dualities\\ and discriminator varieties}

\author[B. A. Davey]{Brian A. Davey}
  \address[Brian A. Davey]{Department of Mathematics and Statistics\\La Trobe University\\
    Victoria 3086, Australia}
  \email{B.Davey@latrobe.edu.au}

\author[A. Gair]{Asha Gair}
  \address[Asha Gair]{Department of Mathematics and Statistics\\La Trobe University\\
    Victoria 3086, Australia}
  \email{agair@students.latrobe.edu.au}

\subjclass[2010]{%
  Primary: 06D50,  
  Secondary: 
  			18A40, 
  			08A40, 
			06D30}    
		
\keywords{Priestley duality, quasi-primal algebra, semi-primal algebra, discriminator variety, Cornish algebra, Ockham algebra}


\begin{abstract}
Anyone who has ever worked with a variety~$\CA$ of algebras with a reduct in the variety of bounded distributive lattices will know a restricted Priestley duality when they meet one---but until now there has been no abstract definition. Here we provide one. After deriving some basic properties of a restricted Priestley dual category $\CX$ of such a variety, we give a characterisation, in terms of $\CX$, of finitely generated discriminator subvarieties of~$\CA$. 

As a first application of our characterisation, we give a new proof of Sankappanavar's characterisation of finitely generated discriminator varieties of distributive double p-algebras.

A substantial portion of the paper is devoted to the application of our results to Cornish algebras. A Cornish algebra is a bounded distributive lattice equipped with a family of unary operations each of which is either an endomorphism or a dual endomorphism of the bounded lattice. They are a natural generalisation of Ockham algebras, which have been extensively studied. We give an external necessary-and-sufficient condition and an easily applied, completely internal, sufficient condition for a finite set of finite Cornish algebras to share a common ternary discriminator term and so generate a discriminator variety. Our results give a characterisation of discriminator varieties of Ockham algebras as a special case, thereby yielding Davey, Nguyen and Pitkethly's characterisation of quasi-primal Ockham algebras. 
\end{abstract}

\maketitle

\section{Introduction}
Hilary Priestley published her famous duality for bounded distributive lattices in 1970~\cite{Pri70}, followed closely by a second paper in 1972~\cite{Pri72}. In 1976, just six years after the publication of Priestley's first paper on the duality, the title \emph{Priestley duality} and  the name \emph{Priestley space}, for an object in the dual category, appeared in the literature. 

Since 1970 a large number of authors have obtained a duality for their favourite variety $\CA$ of bounded-distributive-lattice-based algebras by restricting Priestley duality to $\CA$ and giving explicit descriptions of the Priestley duals of algebras from $\CA$ and of  homomorphisms between algebras in $\CA$. The first two such \emph{restricted Priestley dualities}, for Stone algebras~\cite{Pri74} and for pseudocomplemented distributive lattices~\cite{Pri75}, were published by Priestley herself in 1974 and 1975. In 1984, Priestley published a survey of the first ten years of restricted Priestley dualities~\cite{Pri84}. 

While a list of all restricted Priestley dualities published since 1974 is too long to give here, some of particular significance include:
\begin{enumerate}[ \textbullet]

\item pseudocomplemented distributive lattices: Priestley~\cite{Pri75},

\item Heyting algebras: Esakia~\cite{Esa74}---often referred to as Esakia duality, this duality was obtained directly and later seen to be a restricted Priestley duality (see Davey and Galati~\cite{DG02} for some historical comments),

\item Ockham algebras and its subvarieties: Urquhart~\cite{Urq79}, Cornish and Fowler~\cite{CF77,CF79}, Davey and Priestley~\cite{DP87}, and many others,

\item MV-algebras (or equivalently Wajsberg algebras) and more generally implicative lattices: Mart\'inez~\cite{Mart90}, Martínez and Priestley~\cite{MartPri98},

\item Cornish algebras and its subvarieties: Cornish~\cite{Corn77,Corn86}, Priestley~\cite{Pri97}, Priestley and Santos~\cite{PriSan98}---see Section~\ref{sec:Corn} below.

\end{enumerate}

This paper started out life as an attempt to characterise discriminator varieties of Cornish algebras as an application of their restricted Priestley duality. It soon became clear that the techniques being used could be applied much more generally. Whether they could be applied to all restricted Priestley dualities was unclear as, despite the proliferation of restricted Priestley dualities, there was no overarching theory---no definition of what a restricted Priestley duality actually is. 

In Section~\ref{sec:RestPri} we fill this gap and present an abstract definition of a restricted Priestley duality. We then develop some of the basic properties shared by all restricted Priestley dualities. In particular, we study subalgebras of the product of a pair of algebras from $\CA$, via a restricted Priestley duality, as this is precisely what is needed to study finitely generated discriminator subvarieties of~$\CA$.

Let $A$ be a non-empty set. The map $\tau\colon A^3 \to A$ given by
\[
\tau(x,y,z) :=
 \begin{cases}
 x &\text{if $x \neq y$,}\\
 z &\text{if $x = y$,}
 \end{cases}
\]
is called the \emph{ternary discriminator} on $A$. A variety $\CV$ of algebras is a \emph{discriminator variety} if there is a ternary term $t$ in the language of $\CV$ such that $t^\A$ is the ternary discriminator on every subdirectly irreducible algebra $\A\in \CV$. 
A finite algebra $\A$ is called \emph{quasi-primal} if the ternary discriminator is a term function of~$\A$. (See Theorem~\ref{thm:Quasiprimailty} for several equivalent conditions.) The importance of the ternary discriminator was first recognised by Werner~\cite{W70} and Pixley~\cite{P70,P71}. 
Quasi-primal algebras, and more generally, discriminator varieties play an important role in general algebra; see Burris and Sankappanavar~\cite{BS}, Werner~\cite{W78}, and McKenzie and Valeriote~\cite{McKV89}, for example.

Section~\ref{sec:DiscVar} is devoted to finitely generated discriminator varieties in general. A variety is a finitely generated discriminator variety if and only if it is generated by a finite set $\CB$ of quasi-primal algebras that share a common ternary discriminator term, that is, 
\begin{enumerate}[ \textbullet]

\item there is a ternary term $t$ such that $t^\A$ is the ternary discriminator operation on~$\A$, for all $\A\in \CB$. 
\end{enumerate}
In Theorem~\ref{thm:Quasiprimailty} we give three characterisations of this property. While the theorem is a completely straightforward extension of known charaterisations of quasi-primality, we have not been able to find the result 
in the literature. By combining Theorem~\ref{thm:Quasiprimailty} with the results from the previous section, we give a characterisation of finitely generated discriminator subvarieties of a bounded-distributive-lattice-based variety $\CA$ in terms of a restricted Priestley dual category~$\CX$ for~$\CA$ (Theorem~\ref{thm:qpchar}). We also give a similar characterisation of semi-primal algebras in $\CA$ (Theorem~\ref{thm:spchar}). 

Section~\ref{sec:ddp-algebras} illustrates the application of Theorem~\ref{thm:qpchar} by giving a proof via a restricted Priestley duality  that a variety $\CV$ of distributive double p-algebras is a finitely generated discriminator variety if and only if $\CV$ is generated by a finite set of finite simple distributive double p-algebras. The result is not new (it follows from results of Sankappanavar~\cite{San85}), but the proof is new.

Section~\ref{sec:Corn} is devoted to the study of discriminator varieties of Cornish algebras. 
Cornish algebras are a natural generalisation of Ockham algebras. An Ockham algebra is a bounded distributive lattice equipped with a unary operation $f$ that is a dual endomorphism of the underlying bounded lattice. 
In order to define a Cornish algebra (of type $F$) we replace the single unary operation $f$ in the signature of an Ockham algebra by a set $F= F^+\dotcup F^-$ of unary operations, where $F^+$ is a set of endomorphisms and $F^-$ is a set of dual endomorphisms of the underlying bounded lattice.  Ockham algebras are the algebraic counterpart of the non-classical logic in which the De Morgan laws are retained but the law of the excluded middle and the double negation law are removed. Cornish algebras are the algebraic counterpart of non-classical logics in which we allow more than one De Morgan negation and also allow strong modal operators, like the \emph{next} operator in linear temporal logic. 

Quasi-primal Ockham algebras were recently characterised by Davey, Nguyen and Pitkethly, as part of their investigation of Ockham algebras with finitely many relations~\cite{OckAlg}. In Section~~\ref{sec:Corn} we aim to extend their characterisation of quasi-primal Ockham algebras to a characterisation of discriminator varieties of Cornish algebras. 

We show that in order for there to exist a non-trivial quasi-primal Cornish algebra of type~$F$, we must have $F^- \ne\emptyset$ (Theorem~\ref{thm:ord-pres}). We give an external necessary-and-sufficient condition for a finite set of finite Cornish algebras of type $F$ to share a common ternary discriminator term (Theorem~\ref{thm:qpcharCorn}) and use it to derive a purely internal sufficient condition (Theorem~\ref{thm:mainnew}). These results yield the characterisation of discriminator varieties of Ockham algebras, generalising the characterisation of quasi-primal Ockham algebras from~\cite{OckAlg} (Theorem~\ref{ex:Ock}). They also provide a ready supply of quasi-primal Cornish algebras and discriminator varieties of Cornish algebras (see Example~\ref{ex:Corn} and Example~\ref{ex:Corn2}).

\section{Restricted Priestley dualities}\label{sec:RestPri}

While we all recognise one when we meet one, until now there has been no formal theory of restricted Priestley dualities. We now fill this gap by giving a definition of a restricted Priestley duality and presenting some basic consequences of the definition.

Let $\CA$ be a category of algebras that have a bounded distributive lattice as a (term) reduct. We can obtain a duality for $\CA$ by taking the dual category $\CX$ to be the (usually not full) subcategory of Priestley spaces corresponding to the algebras in $\CA$ and the homomorphisms between them. Such a \emph{restricted Priestley duality} will be most useful if we have an abstract description of the objects and morphisms in~$\CX$. The general set up of such restricted Priestley dualities is described below. We begin with a brief description of Priestley duality in the form we require.

A topological structure $\X =\langle X; \le, \T \rangle$  is a \textit{Priestley space} if $\langle X;\le \rangle$ is an ordered set, $\langle X; \T \rangle$ is a compact topological space, and, for all $x, y \in X$ with $x \nle y$, there is a clopen up-set $V$ with $x \in V$ and $y \notin V$. 

We shall denote the categories of bounded distributive lattices  and Priestley spaces by $\CD$ and~$\CP$, respectively. The morphisms of $\CD$ and~$\CP$ are the natural ones: namely lattice homomorphisms preserving the bounds, and continuous order-preserving maps, respectively. Priestley duality for bounded distributive lattices~\cite{Pri70,Pri72} tells us that these categories are dually equivalent, via the contravariant functors $H \colon \CD \to \CP$ and $K \colon \CP \to \CD$ given in the following two definitions.

\begin{definition}
Let $\TwB = \langle \{0,1\}; \vee, \wedge, 0, 1\rangle$ be the $2$-element bounded lattice and let $\TwT = \langle \{0,1\}; \le, \T\rangle$ be the discretely topologised $2$-element chain. For each bounded distributive lattice $\A  =\langle A; \vee, \wedge, 0, 1\rangle$, define the Priestley space
\[
H(\A) = \langle \CD(\A, \TwB); \le, \T\rangle, 
\]
with order and topology inherited from $\TwT^A$. For each Priestley space~$\X =\langle X; \le, \T \rangle$, define the bounded distributive lattice
\[
K(\X) = \langle \CP(\X, \TwT); \vee, \wedge, 0, 1 \rangle, 
\]
with operations inherited from $\TwB^X$.
\end{definition}

The hom-functors $H$ and $K$ are defined on morphisms in the usual way, as are the natural transformations $e\colon\id\CD \to KH$ and $\varepsilon\colon \id\CP \to HK$.

\begin{definition}\label{def:HKMor}\
\begin{enumerate}[\quad \textbullet]
\item
For each homomorphism $\varphi \colon \A \to \B$ in $\CD$, define $H(\varphi) \colon H(\B) \to H(\A)$ by $H(\varphi)(x) := x \circ \varphi$, for all $x \colon \B \to \TwB$.
\item
For each morphism $\psi \colon \X \to \Y$ in $\CP$, define $K(\psi) \colon K(\Y) \to K(\X)$ by $K(\psi)(\alpha) = \alpha \circ \psi$, for all $\alpha \colon \Y \to \TwT$.
\item
For each $\A\in \CD$, the homomorphism $e_\A \colon \A \to KH(\A)$ is given by evaluation, that is, $e_\A(a)(x) = x(a)$, for all $a \in A$ and $x \colon \A \to \TwB$.
\item
For each $\X\in \CP$, the morphism $\varepsilon_\X \colon \X \to HK(\X)$ is given by evaluation, that is, $\varepsilon_\X(x)(\alpha) = \alpha(x)$, for all $x \in X$ and $\alpha \colon \X \to \TwT$.
\end{enumerate}
\end{definition}

\begin{theorem}[Priestley~\cite{Pri70,Pri72}]
Let $\CD$ and $\CP$ be the categories of bounded distributive lattices and Priestley spaces, respectively.

\begin{enumerate}[\quad \normalfont (1)]

\item $H\colon \CD \to \CP$ and $K\colon \CP \to \CD$ are well-defined functors that yield a dual category equivalence between $\CD$ and $\CP$. 

\item The maps $e_\A \colon \A \to KH(\A)$ and $\varepsilon_\X \colon \X \to HK(\X)$ are isomorphisms, for every $\A\in \CD$ and every $\X\in \CP$.

\end{enumerate}
\end{theorem}

We now give our definition of a restricted Priestley duality. Though much of our theory applies to more general categories, to simplify the development we shall henceforth assume that $\CA$ is a variety of distributive-lattice-based algebras.

\begin{definition}\label{def:restPr}
Let $\CA$ be a variety of algebras that has a term reduct in $\CD$ and let ${}^\flat\colon \CA \to \CD$ be the forgetful functor. Let $\CX$ be a category and let ${}^\flat\colon \CX\to \CP$ be a functor. Assume that $D\colon \CA \to \CX$ and $E\colon \CX\to \CA$ are functors that yield a dual category equivalence between $\CA$ and $\CX$ with unit $e\colon \id\CA \to ED$ and counit $\varepsilon\colon \id\CX \to DE$. We then say that $\langle D, E, e, \varepsilon\rangle$ is a \emph{restricted Priestley duality} between $\CA$ and $\CX$ (with \emph{underlying-Priestley-space functor} ${}^\flat\colon \CX\to \CP$) if 
\begin{enumerate}[ (1)]

\item
the squares in Figure~\ref{fig:RestPries} commute, and 

\item $e_\A^\flat = e_{\A^\flat}$ and $\varepsilon_\X^\flat = \varepsilon_{\X^\flat}$, for all $\A\in \CA$ and all $\X\in \CX$. 
\end{enumerate}
\begin{figure}
\begin{tikzpicture}
 [node distance=1cm,
 auto,
 text depth=0.25ex]
\begin{scope}
\node (A) {$\CA$};
\node (X) [right=of A] {$\CX$};
\node (D) [below=of A] {$\CD$};
\node (P) [right=of D] {$\CP$};
\draw[->,semithick] (A) to node {$D$} (X);
\draw[->,semithick] (A) to node [swap] {$^\flat$} (D);
\draw[->,semithick] (X) to node {$^\flat$} (P);
\draw[->,semithick] (D) to node [swap] {$H$} (P);
\end{scope}
\begin{scope}[xshift=3.5cm]
\node (A) {$\CA$};
\node (X) [right=of A] {$\CX$};
\node (D) [below=of A] {$\CD$};
\node (P) [right=of D] {$\CP$};
\draw[<-,semithick] (A) to node {$E$} (X);
\draw[->,semithick] (A) to node [swap] {$^\flat$} (D);
\draw[->,semithick] (X) to node {$^\flat$} (P);
\draw[<-,semithick] (D) to node [swap] {$K$} (P);
\end{scope}
\end{tikzpicture}
\caption{Restricted Priestley duality}\label{fig:RestPries}
\end{figure}
In particular, as a consequence of (1), we have
\begin{enumerate}[ (1)]

\item[(3)] $D(\A)^\flat = H(\A^\flat)$, for all $\A\in \CA$, and

\item[(4)] $E(\X)^\flat = K(\X^\flat)$, for all $\X\in \CX$.

\end{enumerate} 
It follows from (3) and (4) that $ED(\A)^\flat = KH(\A^\flat)$ and $DE(\X)^\flat = HK(\X^\flat)$. This guarantees that $e_\A^\flat$ and $e_{\A^\flat}$ have the same domain and codomain, as do $\varepsilon_\X^\flat$ and $\varepsilon_{\X^\flat}$, whence Condition~(2) makes sense. It also follows at once from (3) and (4) that the underlying sets of $D(\A)$ and $E(\X)$ are $\CD(\A^\flat, \TwB)$ and $\CP(\X^\flat, \TwT)$, respectively.
 
A morphism $\phi\colon \X \to \Y$ in $\CX$ will be called \emph{$\CP$-surjective} if its underlying Priestley morphism $\phi^\flat\colon \X^\flat \to \Y^\flat$ is surjective. Similarly, $\phi$ is a \emph{$\CP$-embedding} if its underlying Priestley morphism is an embedding of Priestley spaces. 
\end{definition}

\begin{remark}
Given a restricted Priestley duality we are dealing simultaneously with two dual equivalences: $D, E$ between $\CA$ and $\CX$, and $H, K$ between $\CD$ and~$\CP$. Note that, to simplify the notation, we denote the units and counits of the associated dual adjunctions by 
\[
e\colon \id\CA \to ED,\ \varepsilon\colon \id\CX \to DE \quad\text{and}\quad e\colon \id\CD \to KH,\ \varepsilon\colon \id\CP \to HK.
\]
Which of these natural transformations is actually involved will always be clear from the context.
\end{remark}

The following lemma records some easy consequence of this definition.

\begin{lemma} 
Let $\langle D, E, e, \varepsilon\rangle$ be a restricted Priestley duality between $\CA$ and $\CX$.
\begin{enumerate}[\quad \normalfont(1)]

\item The functor $^\flat\colon \CX\to \CP$ is naturally isomorphic to the composite $H\circ {}^\flat\circ E$, and hence is faithful.

\item Both $^\flat \colon \CA \to \CD$ and $\flat\colon \CX \to \CP$ reflect isomorphisms.

\end{enumerate}
\end{lemma}

\begin{proof}
(1) For all $\X\in \CX$, we have 
\[
(H\circ {}^\flat\circ E)(\X) = H(E(\X)^\flat) = HK(\X^\flat).
\]
Hence we may define a natural isomorphism $f\colon {}^\flat \to H\circ {}^\flat\circ E$ by specifying that, $f_\X = \varepsilon_{\X^\flat}\colon \X^\flat  \to (H\circ {}^\flat\circ E)(\X)$, for all $\X\in \CX$, where $\varepsilon\colon \id\CP \to HK$ is the counit of the dual equivalence between $\CD$ and $\CP$.

(2) It is trivial that $^\flat \colon \CA \to \CD$ reflects isomorphisms. Let $\phi$ be a morphism in $\CX$ and assume that $\phi^\flat$ is an isomorphism in $\CP$. By (1) it follows that $(H\circ {}^\flat\circ E)(\phi)$ is an isomorphism. As $E\colon \CX \to \CA$, ${}\flat \colon \CA\to \CD$ and $H\colon \CD \to \CP$ all reflect isomorphisms, it follows that $\phi$ is an isomorphism. Hence $\flat\colon \CX \to \CP$ reflects isomorphisms.
\end{proof}

\begin{lemma}\label{lem:sur-inj}
Let $\langle D, E, e, \varepsilon\rangle$ be a restricted Priestley duality between $\CA$ and $\CX$.
\begin{enumerate}[ \normalfont(1)]

\item A morphism $u\colon \A\to \B$ in $\CA$ is surjective if and only if $D(u)\colon D(\B) \to D(\A)$ is a $\CP$-embedding, and is an embedding if and only if $D(u)\colon D(\B) \to D(\A)$ is $\CP$-surjective.

\item A morphism $\phi\colon \X \to \Y$ in $\CX$ is $\CP$-surjective if and only if $E(\phi)\colon E(\Y) \to E(\X)$ is an embedding, and is a $\CP$-embedding if and only if $E(\phi)\colon E(\Y) \to E(\X)$ is surjective.

\end{enumerate}
\end{lemma}

\begin{proof}
The proof is straightforward. It uses the squares in Figure~\ref{fig:RestPries} along with (a) the trivial fact that $^\flat\colon \CA\to \CD$ preserves and reflects surjectivity and injectivity of homomorphisms, and (b) the well-known property of Priestley duality that morphisms $v\colon \mathbf L\to \mathbf K$ in $\CD$ and  $\psi\colon \mathbb U\to \mathbb V$ in $\CP$ are surjective (are embeddings) if and only if the corresponding morphisms $H(v)\colon H(\mathbf K)\to H(\mathbf L)$ and $K(\psi)\colon K(\mathbb V)\to K(\mathbb U)$ are embeddings (are surjective); see \cite[Theorem 5.19]{ILO}.
\end{proof}

The next two results are simple but important consequences of this lemma. 
Let $\X \in \CX$. A closed subset $\Y$ of the Priestley space $\X^\flat$ is called an \emph{$\CX$-substructure} of~$\X^\flat$ if there is a $\CP$-embedding $\psi\colon \Z \to \X$ in $\CX$ such that $\psi^\flat(\Z^\flat) = \Y$. The ordered set of $\CX$-substructures of $\X^\flat$ is denoted by $\Sub \X$. 

\begin{lemma}\label{lem:Con}   
Let $\langle D, E, e, \varepsilon\rangle$ be a restricted Priestley duality between $\CA$ and $\CX$.
Let $\A\in \CA$ and let $\X=D(\A)$. Then the lattice of congruences on $\A$ is dually isomorphic to $\Sub\X$. 
\end{lemma}

\begin{proof}
Define $u\colon \Con\A \to \Sub\X$ as follows. Let $\eta_\theta \colon \A\to \A/\theta$ be the natural map, for each $\theta \in \Con \A$. Define $\Y_\theta = D(\A/\theta)$; thus, $\psi_\theta := D(\eta_\theta) \colon \Y_\theta \to \X$, which is a $\CP$-embedding by Lemma~\ref{lem:sur-inj}(1). Finally, define $u(\theta) := \psi_\theta^\flat(\Y_\theta^\flat)\subseteq \X^\flat$. It is a straightforward exercise, using Lemma~\ref{lem:sur-inj}, to show that $u$ is a dual order-isomorphism.
\end{proof}

\begin{lemma}\label{lem:factorise}
Let $\langle D, E, e, \varepsilon\rangle$ be a restricted Priestley duality between $\CA$ and $\CX$. Each morphism $\phi\colon \X \to \Y$ in $\CX$ factors in $\CX$ as a $\CP$-surjection $\mu\colon \X\to \Z$ followed by a $\CP$-embedding $\psi\colon \Z \to \Y$. Given two such factorisations $\mu_1\colon \X\to \Z_1$, $\psi_1\colon \Z_1 \to \Y$  and $\mu_2\colon \X\to \Z_2$, $\psi_2\colon \Z_2 \to \Y$  of $\phi$, there is an isomorphism $\gamma \colon \Z_1 \to \Z_2$ in $\CX$ such that $\gamma \circ \mu_1 = \mu_2$ and $\psi_2\circ \gamma = \psi_1$.
\end{lemma}

\begin{proof}
This is a simple consequence of the Lemma~\ref{lem:sur-inj} using the fact that every homomorphism in $\CA$ has a factorisation as a surjective homomorphism followed by an embedding and that this factorisation is unique up to isomorphism.
\end{proof}

The underlying-Priestley-space functor ${}^\flat$ of a restricted Priestley duality usually does not preserve products. Indeed, providing an explicit description of products in the restricted Priestley dual category $\CX$ is a basic research problem as it immediately yields a description of coproducts in~$\CA$. 

Coproducts in $\CX$ are much better behaved: the underlying-Priestley-space functor $^\flat$ both preserves and reflects coproducts.

\begin{lemma}\label{lem:coprod}
Let $\langle D, E, e, \varepsilon\rangle$ be a restricted Priestley duality between $\CA$ and $\CX$. 
\begin{enumerate}[ \normalfont(1)]

\item The underlying Priestley space of the coproduct of $\X_1$ and $\X_2$ in $\CX$ is isomorphic to the disjoint union of their underlying Priestley spaces; more precisely, $(\X_1\sqcup \X_2)^\flat \cong \X_1^\flat \dotcup \X_2^\flat$ via the natural map from $\X_1^\flat \dotcup \X_2^\flat$ to $(\X_1\sqcup \X_2)^\flat$. 

\item Let $\phi_1\colon \X_1 \to \Y$ and $\phi_2\colon \X_2 \to \Y$ be $\CX$-morphisms. Then $\Y$, with the morphisms $\phi_1$ and $\phi_2$, is a coproduct of $\X_1$ and $\X_2$ in~$\CX$ provided $\phi_1$ and $\phi_2$ are $\CP$-embeddings and $\Y^\flat = \phi_1^\flat(\X_1) \dotcup \phi_2^\flat(\X_2)$.

\end{enumerate}
\end{lemma}

\begin{proof}
(1) Using the commuting squares in Figure~\ref{fig:RestPries} along with the fact that $E\colon \CX\to\CA$ sends coproduct to product, ${}^\flat\colon \CA \to \CD$ preserves products, and $H\colon \CD \to \CP$ sends products to disjoint unions, we have:
\begin{multline*}
(\X\sqcup \Y)^\flat \cong \big(D(E(\X\sqcup \Y))\big)^\flat \cong H\big((E(\X)\times E(\Y))^\flat\big) \cong H\big(E(\X)^\flat\times E(\Y)^\flat\big)\\
\cong H(E(\X)^\flat) \dotcup H(E(\Y)^\flat)\cong H(K(\X^\flat)) \dotcup H(K(\Y^\flat))\cong \X^\flat \dotcup \Y^\flat.
\end{multline*}

(2) Assume that $\phi_1$ and $\phi_2$ are $\CP$-embeddings and that $\Y^\flat = \phi_1^\flat(\X_1) \dotcup \phi_2^\flat(\X_2)$. Then $K(\Y^\flat)$, with the morphisms $K(\phi_1^\flat)$ and $K(\phi_2^\flat)$, is a product of $K(\X_1^\flat)$ and $K(\X_2^\flat)$ in~$\CD$. Thus, since $^\flat\circ E = K\circ {}^\flat$, it follows that $E(\Y)^\flat$, with the morphisms $E(\phi_1)^\flat$ and $E(\phi_2)^\flat$, is a product of $E(\X_1)^\flat$ and $E(\X_2)^\flat$ in~$\CD$. Since $^\flat \colon \CA \to \CD$ reflects products, we conclude that $E(\Y)$, with the morphisms $E(\phi_1)$ and $E(\phi_2)$, is a product of $E(\X_1)$ and $E(\X_2)$ in~$\CA$, and so $DE(\Y)$, with the morphisms $DE(\phi_1)$ and $DE(\phi_2)$, is a coproduct of $DE(\X_1)$ and $DE(\X_2)$ in~$\CX$. As $\varepsilon\colon \id\CX\to DE$ is a natural isomorphism, it follows that $\Y$, with the morphisms $\phi_1$ and $\phi_2$, is a coproduct of $\X_1$ and $\X_2$ in~$\CX$.
\end{proof}

We will be using the fact that embeddings into products in $\CA$ are encoded via jointly $\CP$-surjective morphisms in~$\CX$.

\begin{lemma}\label{lem:joint}
Let $\langle D, E, e, \varepsilon\rangle$ be a restricted Priestley duality between $\CA$ and $\CX$.
A homomorphism $u \colon \B \to \A_1 \times \A_2$ in $\CA$ is an embedding if and only if the two morphisms $D(\pi_i \circ u) \colon D(\B_i) \to D(\A)$ in $\CX$, for $i \in \{1, 2\}$, are jointly $\CP$-surjective, where $\pi_i \colon \A_1 \times \A_2 \to \A_i$ denotes the $i$th projection.
\end{lemma}
\begin{proof}
The map $H(\pi_1^\flat) \dotcup H(\pi_2^\flat) \colon H(\A_1^\flat) \dotcup H(\A_2^\flat) \to H(\A_1^\flat \times \A_2^\flat)$ is an isomorphism. Let $u_i := \pi_i \circ u \colon \B \to \A_i$. Then $H(u_i^\flat) = H(\pi_i^\flat \circ u^\flat) = H(u^\flat) \circ H(\pi_i^\flat)$. 
So the diagram in Figure~\ref{fig:jointsurj} commutes.
\begin{figure}[ht]
\begin{center}
\begin{tikzpicture}
  \node[rectangle] (0) at (0,0) {$H(\A_1^\flat) \dotcup H(\A_2^\flat)$};
  \node[rectangle] (1) at (0,1.75) {$H(\A_1^\flat \times \A_2^\flat)$};
  \node[rectangle] (2) at (3,1.75) {$H(\B^\flat)$};
  \path (0) edge [>->>] node {\small $H(\pi_1^\flat) \dotcup H(\pi_2^\flat)$} (1);
  \path (0) edge [->,anchor=north west] node[yshift=0.1cm] {\small $H(u_1^\flat) \dotcup H(u_2^\flat)$} (2);
  \path (1) edge [->] node {\small $H(u^\flat)$} (2);
\end{tikzpicture}
\end{center}
\caption{Diagram for the proof of Lemma~\ref{lem:joint}}\label{fig:jointsurj}
\end{figure}

Thus $H(u^\flat)$ is surjective if and only if $H(u_1^\flat)$ and $H(u_2^\flat)$ are jointly surjective. Consequently, by Lemma~\ref{lem:sur-inj}(1), 
\begin{align*}
&u \colon \B \to \A_1 \times \A_2 \text{ is an embedding}\\
\iff {}&D(u)\colon D(\A_1 \times \A_2) \to D(\B) \text{ is $\CP$-surjective}\\
\iff{}&D(u)^\flat\colon D(\A_1 \times \A_2)^\flat \to D(\B)^\flat \text{ is surjective}\\
\iff{}&H(u^\flat)\colon H(\A_1^\flat \times \A_2^\flat) \to H(\B^\flat) \text{ is surjective}\\
\iff{}&H(u_1^\flat) \colon H(\A_1^\flat) \to H(\B^\flat),\ H(u_2^\flat) \colon H(\A_2^\flat) \to H(\B^\flat) \text{ are jointly surjective}\\
\iff{}&D(u_1) \colon D(\A_1) \to D(\B),\ D(u_2) \colon D(\A_2) \to D(\B) \text{ are jointly $\CP$-surjective,}
\end{align*}
as required.
\end{proof}

The following result will play a vital role in our considerations. Parts~(1) and~(2) have been proved with $\A_1 = \A_2$ in a range of restricted Priestley dualities; see \cite[3.3]{DP93}, \cite[3.5]{DP96}, \cite[p.\ 218]{CD98}, \cite[pp.\ 222--223]{DH04} and~\cite[4.1]{OckAlg}. 

\begin{theorem}\label{lem:subalgofsq}
Let $\langle D, E, e, \varepsilon\rangle$ be a restricted Priestley duality between $\CA$ and $\CX$. Let $\X_1, \X_2,\Y\in \CX$ and define $\A_1 :=E(\X_1)$ and $\A_2 :=E(\X_2)$. 
\begin{enumerate}[ \normalfont(1)]

\item
Let $\phi_1\colon \X_1 \to \Y$, $\phi_2 \colon \X_2\rightarrow\Y$ be jointly $\CP$-surjective morphisms. Then the set 
\[
B(\phi_1, \phi_2) := \{\, (\alpha\circ \phi_1^\flat, \alpha\circ \phi_2^\flat) \mid \alpha \in\CP(\Y^\flat,\TwT)\,\} 
\subseteq \CP(\X_1^\flat,\TwT)\times \CP(\X_2^\flat,\TwT)
\]
forms a substructure of $\A_1\times \A_2$. The corresponding subalgebra $\B(\phi_1, \phi_2)$ of $\A_1\times \A_2$ is isomorphic to $E(\Y)$: the map  $\alpha\mapsto (\alpha\circ \phi_1^\flat, \alpha\circ \phi_2^\flat)$, for all $\alpha \in\CP(\Y^\flat,\TwT)$, is an isomorphism from $E(\Y)$ to $\B(\phi_1, \phi_2)$.

\item
Let $\B$ be a subalgebra of $\A_1\times \A_2$ and let $\Y := D(\B)$. 
For $i \in \{1,2\}$, define morphisms $\varphi_i \colon \X_i \to \Y$ by ${\varphi_i := D(\rho_i) \circ \varepsilon_{\X_i}}$, where $\rho_i\colon\B\to\A_i$ is the projection. Then $\phi_1$ and $\phi_2$ are jointly $\CP$-surjective and satisfy $B = B(\phi_1, \phi_2)$.

\item 
Let $\phi_1\colon \X_1 \to \Y$, $\phi_2 \colon \X_2\rightarrow\Y$ be  jointly $\CP$-surjective morphisms. Then $B(\phi_1, \phi_2) = B_1\times B_2$, for some subalgebras $\B_1$ and $\B_2$ of $\A_1$ and $\A_2$, respectively, if and only if $\Y^\flat = \phi_1^\flat(\X_1^\flat) \dotcup \phi_2^\flat(\X_2^\flat)$. 
\item 
Let $\phi_1\colon \X_1 \to \Y$, $\phi_2 \colon \X_2\rightarrow\Y$ be  jointly $\CP$-surjective morphisms. The subset $B(\phi_1, \phi_2)$ of $A_1\times A_2$ is the graph of a one-to-one partial map from $A_1$ to $A_2$ if and only if both $\phi_1$ and $\phi_2$ are $\CP$-surjective. Moreover, if $\X_1 = \X_2$, then $B(\phi_1, \phi_2)$  is the graph of an identity map on a subset of $A$ if and only if $\phi_1 = \phi_2$ \textup(where $\A = \A_1 = \A_2$\textup). 

\end{enumerate}
\end{theorem}

\begin{proof}
(1) Since $\phi_1\colon \X_1 \to \Y$, $\phi_2 \colon \X_2\rightarrow\Y$ are jointly $\CP$-surjective morphisms, and $(\X_1\sqcup \X_2)^\flat \cong \X_1^\flat \dotcup \X_2^\flat$, by Lemma~\ref{lem:coprod}, the map $\phi:= \phi_1\sqcup \phi_2 \colon \X_1\sqcup \X_2 \to \Y$ is a $\CP$-surjective morphism. The image of the $\CA$-homomorphism 
\[
E(\phi) \colon E(\Y) \to E(\X_1\sqcup \X_2) \cong E(\X_1)\times E(\X_2)
\]
equals the image of the underlying $\CD$-homomorphism
\[
E(\phi)^\flat \colon E(\Y)^\flat \to E(\X_1\sqcup \X_2)^\flat \cong E(\X_1)^\flat\times E(\X_2)^\flat,
\]
which, since ${}^\flat \circ E = K\circ {}^\flat$, equals
\[
K(\phi^\flat) \colon K(\Y^\flat) \to K((\X_1\sqcup \X_2)^\flat) \cong K(\X_1^\flat\dotcup \X_2^\flat) \cong K(\X_1^\flat)\times K(\X_2^\flat).
\]
This map is given by $\alpha \mapsto (\alpha\circ \phi_1^\flat, \alpha\circ \phi_2^\flat)$, for all $\alpha \in\CP(\Y^\flat,\TwT)$, whence its image is $B(\phi_1, \phi_2)$. It follows that $B(\phi_1, \phi_2)$ is a substructure of $\A_1\times \A_2$. Since $\phi$ is $\CP$-surjective, the homomorphism $E(\phi)$ is an embedding, by Lemma~\ref{lem:sur-inj}(2), and hence $E(\Y)$ is isomorphic to $\B(\phi_1, \phi_2)$.

(2) 
Let $\B$ be a subalgebra of $\A_1\times \A_2$ and let $\rho_1\colon\B\to\A_1$, $\rho_2\colon\B\to\A_2$ be the two projections. The inclusion map $\rho_1 \sqcap \rho_2 \colon \B \to \A_1\times \A_2$ is an embedding, and therefore the morphisms $D(\rho_1)\colon D(\A_1) \to D(\B)$ and $D(\rho_2)\colon D(\A_2) \to D(\B)$ are jointly $\CP$-surjective, by Lemma~\ref{lem:joint}.

Since $\A_i = E(\X_i)$, we can define morphisms $\varphi_i \colon \X_i \to D(\B)$ by ${\varphi_i := D(\rho_i) \circ \varepsilon_{\X_i}}$, for $i \in \{1,2\}$. 
As $\varepsilon_{\X_i} \colon \X_i \to DE(\X_i) =D(\A_i)$ is an isomorphism, for $i\in \{1,2\}$, the morphisms $\varphi_1 \colon \X_1 \to D(\B)$ and $\varphi_2 \colon \X_2 \to D(\B)$ are jointly $\CP$-surjective.

Since $e_{\B^\flat} \colon \B^\flat \to KH(\B^\flat)$ is an isomorphism, we have
\begin{align*}
B &= \bigl\{\, \bigl(\rho_1^\flat(b), \rho_2^\flat(b)\bigr) \bigm| b \in B \,\bigr\}\\
   &= \bigl\{\, \bigl( \rho_1^\flat \circ e^{-1}_{\B^\flat} (\alpha),\,  \rho_2^\flat \circ e^{-1}_{\B^\flat} (\alpha) \bigr) \bigm| \alpha \in KH(\B^\flat) \,\bigr\}.
\end{align*}
It remains to check that $\rho_i^\flat \circ e^{-1}_{\B^\flat} (\alpha) = \alpha\circ \varphi_i^\flat $, for each $i \in \{1,2\}$ and $\alpha \in KH(\B^\flat)$. We have $\rho_i^\flat = K(H(\rho_i^\flat) \circ \varepsilon_{\X_i^\flat}) \circ e_{\B^\flat}$, since $\langle H, K, e, \varepsilon \rangle$ is a dual adjunction between $\CD$ and~$\CP$; see~\cite[Figure 1.2]{CD98} or \cite[Exercise 19J(c)]{JoyOfCats}. Thus, using both Condition~(1) and Condition (2) in the definition of a restricted Priestley duality, we have
\begin{multline*}
\rho_i^\flat \circ e^{-1}_{\B^\flat} (\alpha)
 = K(H(\rho_i^\flat) \circ \varepsilon_{\X_i^\flat})(\alpha)
 = \alpha \circ H(\rho_i^\flat) \circ \varepsilon_{\X_i^\flat}\\
 = \alpha \circ D(\rho_i)^\flat \circ \varepsilon_{\X_i}^\flat
 = \alpha \circ (D(\rho_i) \circ \varepsilon_{\X_i})^\flat
 = \alpha \circ \varphi_i^\flat,
\end{multline*}
as required.

(3) 
Let $\phi_1\colon \X_1 \to \Y$, $\phi_2 \colon \X_2\rightarrow\Y$ be  jointly $\CP$-surjective morphisms. First, assume that $\Y^\flat = \phi_1^\flat(\X_1^\flat) \dotcup \phi_2^\flat(\X_2^\flat)$. We have
\[
\CP(\Y^\flat, \TwT) = \big\{ \alpha_1\dotcup \alpha_2 \mid  \alpha_1\in\CP(\phi_1^\flat(\X_1^\flat), \TwT) \And  \alpha_2\in\CP(\phi_2^\flat(\X_2^\flat), \TwT)\big\}.
\]
Hence, by (1), the map $\alpha_1\dotcup \alpha_2\mapsto (\alpha_1\circ \phi_1^\flat, \alpha_2\circ \phi_2^\flat)$, for all $\alpha_1 \in\CP(\phi_1^\flat(\X_1^\flat),\TwT)$ and $\alpha_2 \in\CP(\phi_2^\flat(\X_2^\flat),\TwT)$, is an isomorphism from $E(\Y)$ to $\B(\phi_1, \phi_2)$. Thus, $B(\phi_1, \phi_2) = B(\phi_1) \times B(\phi_2)$, where
\[
B(\phi_i) = \big\{\alpha \circ \phi_i^\flat \mid \alpha \in\CP(\phi_i^\flat(\X_i^\flat),\TwT)\big\},
\] 
for $i\in \{1, 2\}$. Since $\Y^\flat = \phi_1^\flat(\X_1^\flat) \dotcup \phi_2^\flat(\X_2^\flat)$, we have $B(\phi_i) = \pi_i\big(B(\phi_1, \phi_2)\big)$, whence $\B(\phi_i)$ is a subalgebra of $\A_i$, for $i\in \{1, 2\}$. 

Now assume that $\Y^\flat \ne \phi_1^\flat(\X_1^\flat) \dotcup \phi_2^\flat(\X_2^\flat)$. So there exists $x_1\in X_1^\flat$ and $x_2\in X_2^\flat$ such that $y_1:=\phi_1^\flat(x_1)$ is comparable with $y_2:=\phi_2^\flat(x_2)$, say $y_1 \le y_2$. Define $B_i :=\pi_i\big(B(\phi_1, \phi_2)\big)$, for $i\in \{1, 2\}$. We must show that $B(\phi_1, \phi_2)\ne B_1 \times B_2$. Let $\underline 0, \underline 1\in \CP(\Y^\flat, \TwT)$ be the constant maps onto $0$ and $1$, respectively and define $b_1 := \underline 1 \circ \phi_1^\flat$ and $b_2 := \underline 0 \circ \phi_2^\flat$. Then $(b_1, b_2) \in B_1\times B_2$, but $(b_1, b_2) \notin B(\phi_1, \phi_2)$ since $(b_1(x_1), b_2(x_2)) = (1, 0)$, whereas, for all $\alpha\in \CP(\Y^\flat, \TwT)$, we have 
\[
\big((\alpha\circ \phi_1^\flat)(x_1), (\alpha\circ \phi_2^\flat)(x_2)\big) = \big(\alpha(y_1), \alpha(y_2)\big) \ne (1, 0),
\]
since $y_1\le y_2$.

(4) Assume that $\phi_1$ is $\CP$-surjective. We shall show that if $(a, b), (a, c)\in B(\phi_1, \phi_2)$, then $b = c$, that is, $B(\phi_1, \phi_2)$ is the graph of a partial map on $A$. Let $\alpha, \beta\in \CP(\Y,\TwT)$ with 
$\alpha\circ\phi_1^\flat = \beta\circ\phi_1^\flat$. As $\phi_1^\flat$ is surjective, it follows that $\alpha = \beta$; this is just the easy observation that surjective maps are epic. Hence $\alpha\circ\phi_2^\flat = \beta\circ\phi_2^\flat$. By symmetry, if $\phi_2$ is $\CP$-surjective, then $(b, a), (c, a)\in B(\phi_1, \phi_2)$ implies that $b = c$. Hence, if both $\phi_1$ and $\phi_2$ are $\CP$-surjective, then $B(\phi_1, \phi_2)$ is the graph of a one-to-one partial map from $A_1$ to $A_2$.

Suppose now that $\phi_1$ is not $\CP$-surjective, that is, $\phi_1^\flat$ is not surjective. As surjective morphisms correspond to embeddings under Priestley duality, it follows that $K(\phi_1^\flat) \colon K(\Y^\flat) \to K(\X_1^\flat)$ is not one-to-one, that is, there exist $\alpha, \beta\in \CP(\Y^\flat,\TwT)$ with $\alpha\ne \beta$ but 
\[
\alpha\circ\phi_1^\flat = K(\phi_1^\flat)(\alpha) = K(\phi_1^\flat)(\beta)= \beta\circ\phi_1^\flat.
\]
 As $\alpha \ne \beta$, there exists $y\in Y^\flat$ such that $\alpha(y) \ne \beta(y)$. Since $\alpha\circ\phi_1^\flat = \beta\circ\phi_1^\flat$, we must have $y\notin \phi_1^\flat(X_1^\flat)$. Hence there exists $x\in X_2^\flat$ with $\phi_2^\flat(x) = y$, as $\phi_1^\flat$ and $\phi_2^\flat$ are jointly surjective. Thus $\alpha(\phi_2^\flat(x)) \ne \beta(\phi_2^\flat(x))$. Hence, we have $\alpha\circ\phi_1^\flat = \beta\circ\phi_1^\flat$ and $\alpha\circ\phi_2^\flat \ne \beta\circ\phi_2^\flat$, whence $B(\phi_1, \phi_2)$ is not the graph of a partial map from $A_1$ to~$A_2$. By symmetry, it follows that if either $\phi_1^\flat$ or $\phi_2^\flat$ is not surjective, then $B(\phi_1, \phi_2)$ is not the graph of a one-to-one partial map from $A_1$ to $A_2$.

Now assume that $\X_1 = \X_2$ and let $\A = \A_1 = \A_2$. If $\phi_1 = \phi_2$, then of course $B(\phi_1, \phi_2) \subseteq \{\, (a,a)\mid a\in A\,\}$. Assume that $\phi_1 \ne \phi_2$. Then $\phi_1^\flat \ne \phi_2^\flat$ as $^\flat$ is faithful. Hence there exists $x\in X_1^\flat = X_2^\flat$ with $\phi_1^\flat(x)\ne \phi_2^\flat(x)$. Choose $\alpha \in \CP(\Y^\flat,\TwT)$ such that $\alpha(\phi_1^\flat(x))\ne \alpha(\phi_2^\flat(x))$. Then $(\alpha\circ\phi_1^\flat, \alpha\circ\phi_2^\flat)\in B(\phi_1, \phi_2)\comp \{\, (a,a)\mid a\in A\,\}$. Hence $B(\phi_1, \phi_2) \subseteq \{\, (a,a)\mid a\in A\,\}$ if and only if $\phi_1 = \phi_2$.
\end{proof}

\section{Discriminator varieties}\label{sec:DiscVar}

In this section we give a very brief introduction to discriminator varieties, with an emphasis on the finitely generated case. The crucial piece of theory is the characterisation, given in Theorem~\ref{thm:Quasiprimailty}, of finitely generated discriminator varieties. While the result is a completely straightforward generalisation of known results about quasi-primal algebras, it does not appear to have been recorded before. 
Werner's monograph~\cite{W78} gives a detailed treatment of discriminator varieties and quasi-primal algebras. 

We say that a family $\CB $ of algebras \emph{share a common ternary discriminator term} if there is a ternary term $t$ such that $t^\A$ is the ternary discriminator operation on~$\A$, for all $\A\in \CB $.

\begin{facts}\label{facts}
We record some familiar facts about discriminator varieties. 
\begin{enumerate}[ (a)]

\item The variety generated by a quasi-primal algebra is a discriminator variety. 

\item More generally, if $\CB $ is a finite set of finite algebras that share a common ternary discriminator term, then the variety $\Var\CB $ generated by $\CB $ is a discriminator variety, and a non-trivial  algebra $\B\in \Var\CB $ is simple if and only if it is isomorphic to a subalgebra of some $\A\in \CB $.

\item Every finitely generated discriminator variety arises as $\Var\CB $ as described in~(b). 

\item\label{factD} The simple algebras in a discriminator variety form a first-order class. Hence, they either form a proper class or there is a finite bound on their size.

\end{enumerate}
\end{facts}

We begin with a theorem that extends several well-known characterisations of quasi-primal algebras to finite sets of algebras. First, we require some definitions.

Let $n\in \mathbb N$ and let $u_A \colon A^n \to A$ and $u_B \colon B^n \to B$. We say that \emph{$(u_A, u_B)$ preserves a relation $r\subseteq A \times B$} if, for all $a_1, \dots, a_n\in A$ and $b_1, \dots, b_n\in B$, 
\[
(a_1, b_1), \dots, (a_n, b_n) \in r \implies (u_A(a_1, \dots, a_n), u_B(b_1, \dots, b_n)) \in r.
\]
For example, if there is an $n$-ary term $t$ such that $u_A = t^\A$ and $u_B = t^\B$ and $r$ is a subuniverse of $\A\times \B$, then $(u_A, u_B)$ preserves $r$. 
By a \emph{partial isomorphism} between $\A$ and $\B$ we mean an isomorphism from a subalgebra of $\A$ to a subalgebra of~$\B$. 

\begin{theorem}
\label{thm:Quasiprimailty}
Let $\CB $ be a finite set of  finite algebras of signature $F$. Then the following are equivalent:
\begin{enumerate}[ \normalfont(1)]
\item
the algebras in $\CB $ are quasi-primal and share a common ternary discriminator term;

\item
for all $n\in \mathbb N$, if $\{u_A \colon A^n \to A \mid \A \in \CB \}$ is a family of $n$-ary operations indexed by $\CB $ such that, for all $\A_1, \A_2\in \CB $, the pair $(u_{A_1}, u_{A_2})$ preserves the graph of every partial isomorphism between $\A_1$ and $\A_2$, then there is a term $t$ in the  signature $F$ such that $u_A = t^\A$, for all $\A \in \CB $.

\item
$\CB $ generates a congruence-distributive and congruence-permutable
variety and, for all $\A\in \CB $,  every non-trivial subalgebra
of $\A$ is simple;
\item
$\CB $ has a majority term and, for all $\A_1, \A_2 \in \CB $, every subalgebra of $\A_1 \times \A_2$ is either
the product of a subalgebra of $\A_1$ and a subalgebra of $\A_2$ or is the graph of a partial isomorphism between $\A_1$ and $\A_2$.
\end{enumerate}
\end{theorem}

\begin{proof}
For a cyclic proof of the theorem in the case that $\CB  =\{\A\}$, see~\cite[Theorem 3.3.12]{CD98}. It is very easy to see that the proof given in~\cite{CD98} generalises immediately to a proof of Theorem~\ref{thm:Quasiprimailty}---in the proof of (4) $\Rightarrow$ (2), we use the Multisorted NU Duality Theorem~\cite[Theorem 7.1.1]{CD98} rather than the NU Duality Theorem~\cite[Theorem 2.3.4]{CD98}. 
\end{proof}

The following historical remarks about the conditions in Theorem~\ref{thm:Quasiprimailty} in the case that $\CB  =\{\A\}$ are in order: 
\begin{enumerate}[ (a)]

\item Condition~(2) is Pixley's original definition of a quasi-primal algebra~\cite{P71}, the equivalence of~(1) and~(2) is~\cite[Theorem 3.2]{P71}, and the equivalence of~(2) and~(3) is~\cite[Theorem 5.1]{P70}; 

\item the final characterisation given in~(4) is an easy consequence of the Baker--Pixley Theorem~\cite{BP:nu} characterising finite algebras with a near-unanimity term and seems to have been used first in Davey, Schumann and Werner's paper on quasi-primal Helaus~\cite{DSW:91}. 

\end{enumerate}

The next result is an immediate corollary of the theorem. It can also be proved directly by induction on the size of the set~$\CB $, and has a natural generalisation to the situation where $\CB$ has a $(k{+}1)$-ary near unanimity term: replace $1, 2$ by $1, 2, \dots, k$.

\begin{corollary}\label{cor:qpcor}
Let $\CB $ be a finite set of  finite algebras of signature $F$ and assume that $\CB $ has a majority term. The algebras in $\CB $ share a common ternary discriminator term if and only if $\A_1$ and $\A_2$ share a common ternary discriminator term, for all $\A_1, \A_2\in \CB $.
\end{corollary}

The following theorem is an immediate consequence of Theorem~\ref{thm:Quasiprimailty} and Theorem~\ref{lem:subalgofsq}. 

\begin{theorem}\label{thm:qpchar}
Let $\langle D, E, e, \varepsilon\rangle$ be a restricted Priestley duality between $\CA$ and $\CX$.
Let $\CB$ be a finite set of finite algebras from $\CA$ and let $\CY = \{D(\A)\mid \A\in \CB\}$. Then the following are equivalent:

\begin{enumerate}[ \normalfont(1)]

\item the algebras in $\CB$ are quasi-primal and share a common ternary discriminator term;

\item for all $\X_1, \X_2\in \CY$, for each $\Y\in \CX$ and every pair $\phi_1\colon \X_1\to \Y$, $\phi_2\colon \X_2 \to \Y$ of jointly $\CP$-surjective morphisms, we have either $\Y^\flat = \phi_1^\flat(\X_1^\flat) \dotcup \phi_2^\flat(\X_2^\flat)$ or both $\phi_1$ and $\phi_2$ are $\CP$-surjective;

\end{enumerate}
\end{theorem}

A finite algebra is \emph{semi-primal} if it has a majority term and every subalgebra of $\A^2$ is either a product of subalgebras of $\A$ or is the graph $\{\, (b,b)\mid b\in B\,\}$ of the identity function on a subalgebra $\B$ of~$\A$. (This is not Foster and Pixley's original definition~\cite{FP64}, but is equivalent to it.) A characterisation of semi-primal algebras in the presence of a restricted Priestley duality is an easy consequence of the previous theorem.

\begin{theorem}\label{thm:spchar}
Let $\langle D, E, e, \varepsilon\rangle$ be a restricted Priestley duality between $\CA$ and $\CX$.
Let $\A$ be a finite algebra in $\CA$ and let $\X = D(\A)$. Then the following are equivalent:

\begin{enumerate}[ \normalfont(1)]

\item $\A$ is semi-primal;

\item 
for each $\Y\in \CX$ and every pair $\phi_1\colon \X\to \Y$, $\phi_2\colon \X \to \Y$ of jointly $\CP$-surjective morphisms, we have either $\Y^\flat = \phi_1^\flat(\X^\flat) \dotcup \phi_2^\flat(\X^\flat)$ or $\phi_1 = \phi_2$.

\end{enumerate}

\end{theorem}

\begin{proof}
The equivalence follows from the `Moreover' of Theorem~\ref{lem:subalgofsq}(4).
\end{proof}

\section{Distributive double p-algebras}\label{sec:ddp-algebras}

We shall use Theorem~\ref{thm:qpchar} to characterise the finitely generated discriminator varieties of distributive double p-algebras. The result is not new, but the proof is.

An algebra $\A = \langle A; \vee, \wedge, {}^*, {}^+, 0, 1\rangle$ is a \emph{distributive double p-algebra} (ddp-algebra) if $\A^\flat := \langle A; \vee, \wedge, 0, 1\rangle$ is a bounded distributive lattice and ${}^*$ and ${}^+$ are unary operations satisfying
\[
x\wedge y = 0 \iff y\le x^* \quad \text{and} \quad x\vee y = 1 \iff y\ge x^+.
\]
The class $\CA$ of all ddp-algebras is a variety.
The restricted Priestley dual category $\CX$ for $\CA$ was described by Priestley~\cite{Pri75}. Since we are working at the finite level, we do not require the precise description of Priestley spaces dual to a ddp-algebras; it suffices to know that every finite ordered set arises as such a dual. We do require the description at the finite level of the \emph{ddp-space morphisms}, that is, the morphisms dual to homomorphisms between finite ddp-algebras. For each element $x$ of a Priestley space~$\X$, let $\max(x)$ denote the set of maximal elements of $\X$ that dominate $x$, and define $\min(x)$ dually. A map $\phi\colon \X \to \Y$ between finite ordered sets is ddp-space morphism if and only if it is order-preserving and $\phi(\max(x)) = \max(\phi(x))$ and $\phi(\min(x)) = \min(\phi(x))$, for all $x\in X$~\cite{Pri75}. Thus $\CX_{\mathrm{fin}}$, the finite part of $\CX$, consists of all finite ordered sets with the ddp-space morphisms between them. 

We can now state and prove the characterisation of finitely generated discriminator varieties of ddp-algebras. The characterisation follows from results of Sankappanavar~\cite{San85}, who described certain discriminator varieties of ddp-algebras. Recently, the complete characterisation of discriminator varieties of ddp-algebras was completed by Taylor~\cite{Tay16}. Note that not every discriminator variety of ddp-algebras is finitely generated. 

\begin{theorem}
Let $\CB$ be a finite set of finite ddp-algebras and consider the corresponding set $\CY = \{H(\A)\mid \A\in \CB\}$ of finite ordered sets. The following are equivalent:
\begin{enumerate}[ \normalfont(1)]

\item the algebras in $\CB$ are quasi-primal and share a common ternary discriminator term;

\item each $\A\in \CB$ is simple;

\item each $\A\in \CB$ is directly indecomposable and regular, that is, \[
\A\models [a^* = b^* \And a^+ = b^+] \rightarrow a = b; 
\]

\item each $\X\in \CY$ is connected and every element of $\X$ is either maximal or minimal.

\end{enumerate}

\end{theorem}

\begin{proof}
The equivalence of (2), (3) and (4) was established by Davey~\cite{BD78}. As every quasi-primal algebra is simple it remains to show that (4) implies (1). 

Assume that (4) holds. By Theorem~\ref{thm:qpchar}, we must show that if $\Y$ is a finite ordered set and $\phi_1\colon \X_1\to \Y$, $\phi_2\colon \X_2\to \Y$ are jointly surjective ddp-space morphisms with $\X_1, \X_2\in \CY$, then either $\Y = \Y_1 \dotcup \Y_2$, where $\Y_i := \phi_i(\X_i)$, for $i = 1, 2$, or both $\phi_1$ and $\phi_2$ are surjective. Let $\phi_1\colon \X_1\to \Y$, $\phi_2\colon \X_2\to \Y$ be  jointly surjective ddp-space morphisms with $\X_1, \X_2\in \CY$ and assume that $\Y \ne \Y_1 \dotcup \Y_2$. As $Y = Y_1 \cup Y_2$, it follows that there exist comparable elements $a$ and $b$ with $a\in Y_1$ and $b\in Y_2$. Since $\X_1$ and $\X_2$ are connected, both $\Y_1$ and $\Y_2$ are connected and consequently $\Y$ is connected. 

By symmetry, it suffices to show that $\phi_1$ is surjective. As $\phi_1$ and $\phi_2$ map maximal elements to maximal elements and similarly for minimal elements, it follows from (4) that every element of $\Y$ is either maximal or minimal. Let $c\in Y_1$ and let $d\in Y$. Using the fact that $\phi_1(\max(x)) = \max(\phi_1(x))$ and $\phi_1(\min(x)) = \min(\phi_1(x))$, a simple induction on the minimum length of a fence from $c$ to $d$ shows that $d\in Y_1$, whence $\phi_1$ is surjective.
\end{proof} 

\begin{corollary}
A finitely generated variety of ddp-algebras is a discriminator variety if and only if it is generated by a finite set of simple algebras. 
\end{corollary}

\begin{remark}
Given a finite set of finite simple ddp-algebras, an explicit discriminator term for the variety they generate can be obtained from the results of Sankappanavar~\cite[p.\ 413]{San85}. His use of the Heyting implication in his term is justified by a result of Katri{\v{n}}{\'a}k~\cite{Kat73} that every regular ddp-algebra has a term-definable Heyting implication.
\end{remark}

\section{Cornish algebras}\label{sec:Corn}

In this final section, we apply our results to the characterisation of discriminator varieties of Cornish algebras.
As mentioned in the introduction, Cornish algebras are a natural generalisation of Ockham algebras.

More formally, an algebra $\Ok=\langle O; \vee, \wedge, f, 0, 1 \rangle$ is an \emph{Ockham algebra} if its reduct $\Ok^\flat= \langle O; \vee, \wedge, 0, 1\rangle$ is a bounded distributive lattice and $f$ is a dual endomorphism of $\Ok^\flat$. To define a Cornish algebra we must expand the signature by replacing the unary operation $f$ with a set $F$ of unary operations. Let $F=F^+\dotcup F^-$ be a set of unary operation symbols. Then an algebra $\A =\langle A; \vee, \wedge, F^\A, 0, 1\rangle$ is called a \textit{Cornish algebra of type $F$} if 
\begin{enumerate}[\quad \textbullet]
\item
$\Adown :=\langle A; \vee, \wedge, 0, 1\rangle$ is a bounded distributive lattice, and
\item
$F^\A =\{\, f^\A \mid f\in F\,\}$ is a set of unary operations on $A$ such that $f^\A$ is an endomorphism of~$\Adown$, for each $f\in F^+$, and 
$f^\A$ is a dual endomorphism of~$\Adown$, for each $f\in F^-$.
\end{enumerate}
Ockham algebras are the  special case of Cornish algebras in which $F = F^- = \{f\}$. 

Cornish algebras are named after William H. Cornish who introduced them in an invited lecture entitled \emph{Monoids acting on distributive lattices} at the annual meeting of the Australian Mathematical Society at La Trobe University in May 1977. The notes from that lecture were never published but were distributed privately. They first appeared in print as part of Cornish's far-reaching, but often overlooked, monograph~\cite{Corn86} published nine years later. Cornish's monograph contains a wealth of information about Cornish algebras, usually as particular cases of his general theory. Special cases of Cornish's general results have subsequently been published by other authors unaware of the existence of Cornish's work. A detailed analysis of finitely generated varieties of Cornish algebras was given by Priestley~\cite{Pri97} and Priestley and Santos~\cite{PriSan98}.

We begin by describing the restricted Priestley duality for Cornish algebras of type~$F$.
The dual of a Cornish algebra will be a Priestley space equipped with a family of continuous self maps each of which is either order-preserving or order-reversing. The duality was first described by Cornish~\cite{Corn77, Corn86} and follows easily from Urquhart's restricted Priestley duality for Ockham algebras~\cite{Urq79}. 

A topological structure $\X :=\langle X; F^\X, \le, \T \rangle$  is a \textit{Cornish space of type $F$} if
\begin{itemize}
\item
$\X^\flat :=\langle X;\le, \T \rangle$ is a Priestley space, and
\item
$F^\X =\{\, f^\X \mid f\in F\,\}$ is a set of unary operations on $X$ such that $f^\X$ is a continuous order-preserving self-map of~$\X^\flat$, for each $f\in F^+$, and $f^\X$ is a continuous order-reversing self-map of $\X^\flat$, for each $f\in F^-$. 
\end{itemize}
\emph{Ockham spaces} arise in the special case when $F = F^- = \{f\}$. 

Except in the examples at the end, throughout the remainder of this section, we fix a type $F$ and will sometimes say simply Cornish algebra or Cornish space without explicit reference to the type.

We shall denote the categories of Cornish algebras and Cornish spaces of type~$F$ by $\CCF$ and~$\CXF$, respectively. The morphisms of $\CCF$ and $\CXF$ are the natural ones: namely Cornish-algebra homomorphisms, and continuous order-preserving maps that preserve the operations in~$F$, respectively. These categories are dually equivalent, via the contravariant functors $D \colon \CCF \to \CXF$ and $E \colon \CXF \to \CCF$ given on objects in the following definition.

\begin{definition}
Let $c\colon \{0, 1\} \to \{0, 1\}$ be the usual Boolean complementation: $c(0) = 1$ and $c(1) = 0$. For each Cornish algebra~$\A  =\langle A; \vee, \wedge, F^\A, 0, 1\rangle$, define the Cornish space
\[
D(\A) = \langle \CD(\Adown, \TwB); F^{D(\A)}, \le, \T\rangle,
\]
where $\langle \CD(\Adown, \TwB); \le, \T\rangle$ is the Priestley space dual to the underlying bounded distributive lattice~$\Adown$ and, for all $f\in F$, the unary operation $f^{D(\A)}$ is given by 
\[
f^{D(\A)}(x) =\begin{cases}
			 x \circ f^\A, &\text{ if $f\in F^+$,}\\
			 c\circ x \circ f^\A, &\text{ if $f\in F^-$,}
			 \end{cases}
\]
for all $x \colon \Adown \to \TwB$.

For each Cornish space~$\X =\langle X; F^\X, \le, \T \rangle$, define the Cornish algebra
\[
E(\X) = \langle \CP(\X^\flat, \TwT); \vee, \wedge, F^{E(\X)}, 0, 1 \rangle, 
\]
where $\langle \CP(\X^\flat, \TwT); \vee, \wedge, 0, 1 \rangle$ is the bounded distributive lattice dual to the Priestley space~$\X^\flat$, and, for all $f\in F$, the unary operation $f^{E(\X)}$ is given by 
\[
f^{E(\X)}(\alpha) =\begin{cases}
				 \alpha \circ f^\X, &\text{ if $f\in F^+$,} \\
				 c\circ \alpha \circ f^\X, &\text{ if $f\in F^-$,}
				 \end{cases}
\]
for all $\alpha \colon \X^\flat \to \TwT$. 
\end{definition}

The hom-functors $D$ and $E$ are defined on morphisms in exactly the way the functors $H$ and $K$ were defined in Definition~\ref{def:HKMor}. The same is true of the natural transformations $e\colon\id\CCF \to ED$ and $\varepsilon\colon \id\CXF \to DE$.

\begin{theorem}[Restricted Priestley duality for Cornish algebras]
Let $\CCF$ and $\CXF$ be, respectively, the categories of Cornish algebras and Cornish spaces of type $F$. 

\begin{enumerate}[\quad \normalfont (1)]

\item Then $D\colon \CCF \to \CXF$ and $E\colon \CXF \to \CCF$ are well-defined functors that yield a dual category equivalence between $\CCF$ and $\CXF$. 

\item The maps $e_\A \colon \A \to ED(\A)$ and $\varepsilon_\X \colon \X \to DE(\X)$ are isomorphisms, for every Cornish algebra~$\A$ and every Cornish space~$\X$ of type $F$.

\end{enumerate}
\end{theorem}

\begin{remark}
An ordered set $\X$ can have self maps $f^\X$ that are both order-preserving and order-reversing. Of course, the unary operation $f^{E(\X)} := E(f^\X)$ on the corresponding Cornish algebra will change radically depending upon whether $f\in F^+$ or $f\in F^-$.
 The examples in Figure~\ref{tab:mapcolour} illustrates this.
\end{remark}
\begin{figure}[ht]
\begin{tabular}{|c@{\hspace*{2em}}cc|}
  \cline{1-3}
 Ordered set & \multicolumn{2}{c|}{The corresponding algebra $E(\X)$} \\
    \cline{2-3}
  with map $f^\X$& $f\in F^+$ & $f\in F^-$\\
  \hline
\begin{tikzpicture}
   \useasboundingbox (-1,-0.25) rectangle (1,2.4);
   \begin{scope}[yshift=0.5cm]
     \node[unshaded] (0) at (0,0) {};
     \node[unshaded] (1) at (0,1) {};
     \draw[order] (0) to (1);
     \draw[loopy] (1) to [out=205, in=165] (1);
     \draw[curvy] (0) to [bend left] (1);
     \end{scope}
\end{tikzpicture}
&
\begin{tikzpicture}
   \useasboundingbox (-1,-0.25) rectangle (1,2.4);
   \begin{scope}
     \node[unshaded] (2) at (0,0) {};
     \node[unshaded] (3) at (0,1) {};
     \node[unshaded] (4) at (0,2) {};
     \draw[order] (2) to (3);
     \draw[order] (3) to (4);
     \draw[loopy] (2) to [out=205, in=165] (2);
     \draw[loopy] (4) to [out=205, in=165] (4);
     \draw[curvy] (3) to [bend left] (4);
     \end{scope}
\end{tikzpicture}
&
\begin{tikzpicture}
   \useasboundingbox (-1,-0.25) rectangle (1,2.4);
   \begin{scope}
     \node[unshaded] (2) at (0,0) {};
     \node[unshaded] (3) at (0,1) {};
     \node[unshaded] (4) at (0,2) {};
     \draw[order] (2) to (3);
     \draw[order] (3) to (4);
     \draw[curvy, dotted] (3) to [bend left] (2);
     \draw[curvy, dotted, <->] (2) to [bend right=60, min distance=0.7cm] (4);
     \end{scope}
\end{tikzpicture}\\
\hline
\begin{tikzpicture}
   \useasboundingbox (-1,-0.25) rectangle (1,1.75);
   \begin{scope}[xshift=-0.5cm,yshift=0.71cm]
     \node[unshaded] (5) at (0,0) {};
     \node[unshaded] (6) at (1,0) {};
     \draw[curvy] (6) to [bend left] (5);
     \draw[curvy] (5) to [bend left] (6);
     \end{scope}
\end{tikzpicture}
&
\begin{tikzpicture}
  \useasboundingbox (-1,-0.25) rectangle (1,1.75);
   \begin{scope}
     \node[unshaded] (7) at (0,0) {};
     \node[unshaded] (a) at ($(7) + (135:1)$) {};
     \node[unshaded] (b) at ($(7) + (45:1)$) {};
     \node[unshaded] (8) at ($(a) + (45:1)$) {};
     \draw[order] (7) to (a);
     \draw[order] (7) to (b);
     \draw[order] (a) to (8);
     \draw[order] (b) to (8);
     \draw[straight, <->] (a) to (b);
     \draw[loopy] (7) to [out=205, in=165] (7);
     \draw[loopy] (8) to [out=205, in=165] (8);
   \end{scope}
\end{tikzpicture}
&
\begin{tikzpicture}
   \useasboundingbox (-1.5,-0.25) rectangle (1.5,1.75);
   \begin{scope}
     \node[unshaded] (13) at (0,0) {};
     \node[unshaded] (g) at ($(13) + (135:1)$) {};
     \node[unshaded] (h) at ($(13) + (45:1)$) {};
     \node[unshaded] (14) at ($(g) + (45:1)$) {};
     \draw[order] (13) to (g);
     \draw[order] (13) to (h);
     \draw[order] (g) to (14);
     \draw[order] (h) to (14);
     \draw[straight, dotted, <->] (13) to (14);
     \draw[loopy, dotted] (g) to [out=205, in=165] (g);
     \draw[loopy, dotted] (h) to [out=25,in=-15] (h);
   \end{scope}
\end{tikzpicture}\\
\hline
\begin{tikzpicture}
   \useasboundingbox (-1,-0.25) rectangle (1,1.75);
   \begin{scope}[xshift=-0.5cm,yshift=0.71cm]
     \node[unshaded] (9) at (0,0) {};
     \node[unshaded] (10) at (1,0) {};
     \draw[loopy] (9) to [out=115,in=75] (9);
     \draw[loopy] (10) to [out=115,in=75] (10);
   \end{scope}
\end{tikzpicture}
&
\begin{tikzpicture}
   \useasboundingbox (-1.5,-0.25) rectangle (1.5,1.75);
   \begin{scope}
     \node[unshaded] (11) at (0,0) {};
     \node[unshaded] (e) at ($(11) + (135:1)$) {};
     \node[unshaded] (f) at ($(11) + (45:1)$) {};
     \node[unshaded] (12) at ($(e) + (45:1)$) {};
     \draw[order] (11) to (e);
     \draw[order] (11) to (f);
     \draw[order] (e) to (12);
     \draw[order] (f) to (12);
     \draw[loopy] (11) to [out=205, in=165] (11);
     \draw[loopy] (12) to [out=205, in=165] (12);
     \draw[loopy] (e) to [out=205, in=165] (e);
     \draw[loopy] (f) to [out=25,in=-15] (f);
   \end{scope}
\end{tikzpicture}
&
\begin{tikzpicture}
   \useasboundingbox (-1,-0.25) rectangle (1,1.75);
   \begin{scope}
     \node[unshaded] (7) at (0,0) {};
     \node[unshaded] (c) at ($(7) + (135:1)$) {};
     \node[unshaded] (d) at ($(7) + (45:1)$) {};
     \node[unshaded] (8) at ($(c) + (45:1)$) {};
     \draw[order] (7) to (c);
     \draw[order] (7) to (d);
     \draw[order] (c) to (8);
     \draw[order] (d) to (8);
     \draw[straight, dotted, <->] (c) to (d);
     \draw[straight, dotted, <->] (7) to (8);
   \end{scope}
\end{tikzpicture}\\
\hline
\end{tabular}
\caption{$f\in F^+$ versus $f\in F^-$}\label{tab:mapcolour}
\end{figure}
As our first application of the duality, we will now see that $\CCF = \ISP(\C)$ where $\C$ is an algebra of cardinality $2^\kappa$, with $\kappa = \max\{\aleph_0, |F|\}$. Let $\mathcal M$ be the free monoid generated by $F$---concretely, $\mathcal M$ is the set $F^*$ of all finite words in the alphabet $F$, with the empty word $\epsilon$ as identity. Define
\[
\C := \langle \{0, 1\}^\mathcal M; \vee, \wedge, F^\C, 0, 1\rangle,
\]
where $\vee$, $\wedge$, $0$ and $1$ are defined pointwise, and, for all $f\in F$, all $a\in \{0, 1\}^\mathcal M$ and all $w\in \mathcal M$,
\[
f^\C(a)(w) := \begin{cases}
				a(fw), \quad &\text{if $f\in F^+$,}\\
				c(a(fw)), 		&\text{if $f\in F^-$.}
			\end{cases}	
\]
Extend the $+/-$ labelling of $F$ to $\mathcal M$ by defining $\mathcal M^+$ to consist of words containing an even number of symbols from $F^-$ and $\mathcal M^-$ to consist of words containing an odd number of symbols from $F^-$. Then, for each word $v$ in $\mathcal M$, we can define a map $v^\C$ in a way analogous to the definition of $f^\C$; just replace $f\in F^+$ and $f\in F^-$ by $v\in \mathcal M^+$ and $v\in \mathcal M^-$, respectively. Note that the map $v^\C$ is the term function corresponding to $v$ viewed as a unary term and $\epsilon^\C = \id C$.

The following result is a consequence of very general results of Cornish~\cite[8.19.2, 8.20.1]{Corn86}; in his language $\mathcal M$ is a $\pm$-monoid, with $\mathcal M^+$ and $\mathcal M^-$ as just defined. We sketch the details of a direct proof.

\begin{theorem}\label{thm:genM}
The algebra $\C$ defined above is a Cornish algebra of type $F$ and every Cornish algebra of type $F$ embeds into a power of $\C$. Consequently, $\CCF$ is residually small.
\end{theorem}

\begin{proof}
It is easy to check that $f^\C$ is an endomorphism of $\C^\flat$ when $f\in F^+$ and is a dual endomorphism of $\C^\flat$ when $f\in F^-$. Hence $\C$ is a Cornish algebra of type~$F$.

Let $\A = E(\X)$ be a Cornish algebra of type $F$. 
To prove that $\A$ embeds into a power of~$\C$, it suffices to show that the homomorphisms from $\A = E(\X)$ to $\C$ separate the points of $\A$.
For each $x\in X$, we shall define a homomorphism $\phi_x \colon \A \to \C$. Recall that the underlying sets of $\A = E(\X)$ and $\C$ are $\CP(\X^\flat, \TwT)$ and $\{0, 1\}^\mathcal M$, respectively. Define $\phi_x \colon \CP(\X^\flat, \TwT)\to \{0, 1\}^\mathcal M$ as follows: 
\[
\phi_x(\alpha)(w) := \alpha(w^\X(x)), \text{ for all $\alpha \in \CP(\X^\flat, \TwT)$ and all $w\in \mathcal M$.}
\]
A simple check shows that $\phi_x$ is a homomorphism. Now let $\alpha, \beta\in \CP(\X^\flat, \TwT)$ with $\alpha \ne \beta$. Thus, there exists $x\in X$ with $\alpha(x)\ne \beta(x)$. Without loss of generality, we may assume that $\alpha(x) = 1$ and $\beta(x) = 0$, and hence
\begin{alignat*}{2}
\phi_x(\alpha)(\epsilon) &= \alpha(\epsilon^\X(x)) = \alpha(x) = 1,\\
\phi_x(\beta)(\epsilon) &=  \beta(\epsilon^\X(x)) = \beta(x) = 0.
\end{alignat*}
Thus, $\phi_x(\alpha) \ne \phi_x(\beta)$, as required. 

As $\CCF = \ISP(\C)$, every subdirectly irreducible algebra in $\CCF$ embeds into $\C$ and consequently $\CCF$ is residually small.
\end{proof}

The next lemma is an immediate consequence of Lemma~\ref{lem:Con}. 
Note that, given a Cornish space~$\X$, the lattice $\Sub\X$, introduced immediately before Lemma~\ref{lem:Con}, is the lattice of topologically closed substructures of $\X$, including the empty substructure. 

\begin{lemma}\label{lem:sub-alg}   
Let $\A=E(\X)$ with $\X$ a Cornish space of type~$\X$. Then the lattice of congruences on $\A$ is dually isomorphic to $\Sub\X$. In particular, $\A$ is simple if and only if $\X$ has no non-empty proper closed substructures.
\end{lemma}

Our first observation is that there are no non-finitely generated discriminator varieties of Cornish algebras of finite type. 

\begin{theorem}\label{thm:nodiscvars}
Let $\CV$ be a discriminator variety of Cornish algebras of type~$F$ with $F$ finite. Then $\CV = \Var\CA$ for some finite set $\CA$ of quasi-primal Cornish algebras of type $F$ that share a common ternary discriminator term.
\end{theorem}

\begin{proof}
Since $\CV$ is residually small, Fact~\ref{facts}(\ref{factD}) implies that there is a finite bound on the sizes of the simple algebras in $\CV$, and hence, as $F$ is finite, there are only finitely many of them (up to isomorphism). As the simple algebras in a discriminator variety share a common ternary discriminator term and generate the variety, the result follows.
\end{proof}

As a result of this theorem, the task of characterising discriminator varieties of Cornish algebras of finite type reduces to characterising finite sets of finite Cornish spaces satisfying Condition~(2) of Theorem~\ref{thm:qpchar}.

Since a quasi-primal algebra must be simple, it follows at once from Lemma~\ref{lem:sub-alg} that a necessary condition for a finite Cornish algebra $\A = E(\X)$ to be quasi-primal is that the Cornish space $\X$ has no non-empty proper substructures. Our first result gives a less obvious necessary condition for a Cornish algebra to be quasi-primal, namely that $F^-$ is non-empty. The proof is direct and does not rely upon the restricted Priestley duality for Cornish algebras.

\begin{theorem}\label{thm:ord-pres}
Let $\A$ be a non-trivial finite Cornish algebra of type $F$. If $F^-=\emptyset$, then $\A$ is not quasi-primal.
\end{theorem}

\begin{proof}
Let $\A = \langle A; \vee, \wedge, 0, 1, F\rangle$ be a non-trivial finite Cornish algebra of type $F$ and assume that $F^-=\emptyset$. Suppose, by way of contradiction, that $\A$ is quasi-primal and let $t$ be a $3$-ary term that yields the ternary discriminator on $\A$. As every fundamental operation of $\A$ is order-preserving, the term function $t^\A$ is order-preserving. Since $t$ is a Mal'cev term on $\A$ (that is $t^\A(a,a,b) = b$ and $t^\A(a,b,b) = a$, for all $a, b\in A$), the order on $\A$ is an antichain; indeed, for all $a, b \in A$, we have
\[
a\le b \implies b = t^\A(a,a,b) \le t^\A(a,b,b) = a.
\]
This is a contradiction since $\Adown$ is a non-trivial lattice. Hence $\A$ is not quasi-primal.
\end{proof}

We can give an external characterisation, in terms of jointly surjective morphisms in the dual category, of finite sets of quasi-primal Cornish algebras of type $F$ that share a common ternary discriminator term.

Given a Cornish-space morphism $\phi\colon \X\to \Y$, we define $\phi(\X)$ to be the substructure of $\Y$ with underlying set~$\phi(X)$.

\begin{theorem}\label{thm:qpcharCorn}
Let $\CB$ be a finite set of finite Cornish algebras of type~$F$ and let $\CY = \{D(\A)\mid \A\in \CB\}$. Then the following are equivalent:

\begin{enumerate}[ \normalfont(1)]

\item the algebras in $\CB$ are quasi-primal and share a common ternary discriminator term;

\item for all $\X_1, \X_2\in \CY$, for each Cornish space $\Y$ and every pair $\phi_1\colon \X_1\to \Y$, $\phi_2\colon \X_2 \to \Y$ of jointly surjective morphisms, we have either $\Y = \Y_1\dotcup \Y_2$, where $\Y_i := \phi_i(\X_i)$, for $i = 1, 2$, or both $\phi_1$ and $\phi_2$ are surjective;

\item each $\X\in \CX$ has no non-empty proper substructures and, for all $\X_1, \X_2\in \CX$, for each Cornish space $\Y$ and every pair $\phi_1\colon \X_1\to \Y$, $\phi_2\colon \X_2 \to \Y$ of jointly surjective morphisms, if $\phi_1(a)$ and $\phi_2(b)$ are comparable, for some $a\in X_1$ and $b\in X_2$, then there exists $c\in X_1$ and $d\in X_2$ with $\phi_1(c) = \phi_2(d)$.
\end{enumerate}
\end{theorem}
\begin{proof}
(1) $\Leftrightarrow$ (2). This follows immediately from Theorem~\ref{thm:qpchar}.

We will now prove that (1) and (2) together imply (3). Assume (1) and (2). By Theorem~\ref{thm:Quasiprimailty}, each $\A\in \CA$ is simple. Hence, by Lemma~\ref{lem:sub-alg}, each $\X\in \CX$ has no non-empty proper substructures. Let $\X_1, \X_2\in \CY$, let $\Y$ be a Cornish space, let $\phi_1\colon \X_1\to \Y$, $\phi_2\colon \X \to \Y$ be  jointly surjective morphisms and define $\Y_i := \phi_i(\X_i)$, for $i = 1, 2$. Assume that $a\in X_1$ and $b\in X_2$ with $\phi_1(a)$ comparable with $\phi_2(b)$. It follows that $\Y \ne \Y_1\dotcup \Y_2$. Thus, by (2), both $\phi_1$ and $\phi_2$ are surjective. Hence there certainly exist $c\in X_1$ and $d\in X_2$ with $\phi_1(c) = \phi_2(d)$.

(3) $\Rightarrow$ (2). Assume (3). Let $\X_1, \X_2\in \CY$, let $\Y$ be a Cornish space, let ${\phi_1\colon \X_1\to \Y}$, $\phi_2\colon \X \to \Y$ be  jointly surjective morphisms and define $\Y_i := \phi_i(\X_i)$, for $i = 1, 2$.  Assume that $\Y \ne \Y_1\dotcup \Y_2$. We must show that both $\phi_1$ and $\phi_2$ are surjective. Since $\Y \ne \Y_1\dotcup \Y_2$, there exist $a\in X_1$ and $b\in X_2$ with $\phi_1(a)$ comparable with~$\phi_2(b)$, and so, by (3), there exists  $c\in X_1$ and $d\in X_2$ with $\phi_1(c) = \phi_2(d)$. Since $\X_1$ and $\X_2$ have no proper substructures, we have $\sg {\X_1} c = \X_1$ and $\sg {\X_2} d = \X_2$, where $\sg \X x$ denotes the substructure of $\X$ generated by~$x$. It is easy to check that $\phi_i(\sg {\X_i} x) = \sg \Y {\phi_i(x)}$, and hence, since $\phi_1(c) = \phi_2(d)$,  
\[
\phi_1(\X_1) = \phi_1(\sg {\X_1} c) = \sg \Y {\phi_1(c)} = \sg \Y {\phi_2(d)} = \phi_2(\sg {\X_2} d) = \phi_2(\X_2).
\]
Since $\phi_1$ and $\phi_2$ are jointly surjective, it follows that 
\[
\Y = \phi_1(\X_1) \cup \phi_2(\X_2) = \phi_1(\X_1) = \phi_2(\X_2), \]
whence both $\phi_1$ and $\phi_2$ are surjective.
\end{proof}

A characterisation of semi-primal Cornish algebras is an easy consequence of the previous theorem.

\begin{theorem}\label{thm:spcharCorn}
Let $\A$ be a  finite Cornish algebra of type $F$ and let $\X = D(\A)$. Then the following are equivalent:

\begin{enumerate}[ \normalfont(1)]

\item $\A$ is semi-primal;

\item 
for each Cornish space $\Y$ and jointly surjective morphisms $\phi_1, \phi_2\colon \X \to \Y$, either $\Y = \Y_1\dotcup \Y_2$, where $\Y_i := \phi_i(\X)$, or $\phi_1 = \phi_2$.

\item $\X$ has no non-empty proper substructures and, for each Cornish space $\Y$ and jointly surjective morphisms $\phi_1, \phi_2\colon \X \to \Y$, if $\phi_1(a)$ and $\phi_2(b)$ are comparable, for some $a, b\in X$, there exists $c\in X$ with $\phi_1(c) = \phi_2(c)$.
\end{enumerate}

\end{theorem}

\begin{proof}
The equivalence of (1) and (2) follows from Theorem~\ref{thm:spchar}, and (2) implies (3) is trivial. Finally, we prove that (3) implies (2). Assume that (3) holds, let $\phi_1, \phi_2\colon \X \to \Y$ be  jointly surjective morphisms, and define $\Y_i := \phi_i(\X)$. Assume that $\Y \ne \Y_1\dotcup \Y_2$. We must prove that $\phi_1 = \phi_2$. As $\Y \ne \Y_1\dotcup \Y_2$, there exist $a,b\in X$ such that $\phi_1(a)$ and $\phi_2(b)$ are comparable. Then (3)~guarantees that there exists $c\in X$ with $\phi_1(c) = \phi_2(c)$. As $\X$ has no non-empty proper substructures, the element $c$ generates $\X$, whence $\phi_1(c) = \phi_2(c)$ implies that $\phi_1 = \phi_2$.  
\end{proof}

The characterisations of quasi- and semi-primality given by Theorems~\ref{thm:qpcharCorn} and~\ref{thm:spcharCorn} are external to the Cornish spaces involved in the theorems as they require us to study pairs of jointly surjective maps defined on those spaces. We now give sufficient conditions that are local to the Cornish spaces.  

Given a Cornish space $\X$ of type $F$, a unary term $t$ in the signature $F$ and $a\in X$, the \emph{orbit of $a$ under $t^\X$} is the sequence
\[
\text{$a(0) :=a$, $a(1) := t^\X(a)$, $a(2):= t^\X(t^\X(a))$, \dots\ }
\]
of iterates of $a$ under $t^\X$. If the orbit is finite, in particular if $\X$ is finite, the orbit must eventually cycle, that is, there exist $m, n\in \mathbb N_0$ with $n < m$ such that $a(0), \dots ,a(m-1)$ are pairwise distinct and $a(m) = a(n)$, in which case $\{a(n), \dots ,a(m-1)\}$ is a cycle of length $m-n$. We say that \emph{the orbit of $a$ under $t^\X$ eventually reaches an odd cycle} if $m-n$ is odd. We require the following result of Davey, Nguyen and Pitkethly~\cite{OckAlg}.

\begin{lemma}[{\cite[Lemma 4.3]{OckAlg}}]\label{lem:antichain}
 Let $g$ be an order-reversing self map of an ordered set $\Y$. Then the union of the odd cycles of $g$ forms an antichain. 
\end{lemma}

As we noted earlier when discussing the Cornish algebra $\C$ of type $F$ used in Theorem~\ref{thm:genM}, we can extend the $+/-$ labelling of the operation symbols to the set $T$ of unary terms in the signature $F$ in the natural way: $t(v)\in T^+$ if and only if $t(v)$ contains an even number of operation symbols from $F^+$.  

\begin{theorem}\label{thm:mainnew} 
Let $\CB$ be a finite set of finite Cornish algebras of type~$F$ and let $\CY = \{D(\A)\mid \A\in \CB\}$.
Then the algebras in $\CB$ are quasi-primal and share a common ternary discriminator term provided $\CY$ has the following properties:

\begin{enumerate}[ \normalfont(i)]
\item
each $\X\in \CY$ has no non-empty proper substructures,
\item
there exists a unary term $t$ in the signature $F$ 
such that $t\in T^-$ and, for all $\X\in \CY$ and all $a\in X$, the orbit of $a$ under $t^\X$ eventually reaches an odd cycle.

\end{enumerate}
Moreover, if we replace \textup{(ii)} by
\begin{enumerate}[ \normalfont(i)]

\item[\normalfont(ii)$'$] there exists a unary term $t$ in the signature $F$
such that $t\in T^-$ and, for all $\X\in \CY$, the map $t^\X$ is constant \textup(that is, $t^\X(a)=t^\X(b)$, for all $a,b\in X$\textup),

\end{enumerate}
then the algebras in $\CA$ are semi-primal and share a common ternary discriminator term.
\end{theorem}

\begin{proof}
We shall use the equivalence of (1) and (3) in Theorem~\ref{thm:qpcharCorn}.
Assume that (i) and (ii) hold. Let $\X_1, \X_2\in \CY$, let $\phi_1\colon \X_1\to \Y$, $\phi_2\colon \X_2 \to \Y$ be  jointly surjective morphisms, for some Cornish space $\Y$, and assume that $\phi_1(a)$ and $ \phi_2(b)$ are comparable, for some $a\in X_1$ and $b\in X_2$. We need to find $c\in X_1$ and $d\in X_2$ with $\phi_1(c) = \phi_2(d)$.

By (ii), there exists a unary term $t\in T^-$ such that the orbits of $a$ and $b$ under $t^{\X_1}$ and $t^{\X_2}$, respectively, eventually reach odd cycles. Choose $m\in \mathbb N$ large enough so that $c:= (t^{\X_1})^m(a)$ and $d:=(t^{\X_2})^m(b)$ belong to odd cycles of $t^{\X_1}$ and $t^{\X_2}$, respectively. 
We have
\begin{align*}
\phi_1(c) &= \phi_1((t^{\X_1})^m(a)) = (t^\Y)^m(\phi_1(a)), \text{ and}\\
\phi_2(d) &= \phi_2((t^{\X_2})^m(b)) = (t^\Y)^m(\phi_2(b)).
\end{align*}
Hence, since $\phi_1(a)$ and $\phi_2(b)$ are comparable, so are 
$\phi_1(c)$ and $\phi_2(d)$. As $c$ and $d$ belong to odd cycles of $t^{\X_1}$ and $t^{\X_2}$, respectively, it follows that $\phi_1(c)$ and $\phi_2(d)$ belong to odd cycles of~$t^\Y$. By Lemma~\ref{lem:antichain} we have $\phi_1(c) = \phi_2(d)$, as required. As $\X$ has no non-empty proper substructures, by~(i), it follows from Theorem~\ref{thm:qpchar} that the algebras in $\CA$ are quasi-primal and share a common ternary discriminator term.

Finally, assume that (i) and (ii)$'$ hold. As (ii)$'$ implies (ii), it follows that the algebras in $\CA$ share a common ternary discriminator term. It remains to prove that $\A = E(\X)$ is semi-primal, for all $\X\in \CY$. Let $\X\in \CY$ and let $e$ be the constant value of $t^\X$. Then, in the calculations above with $\X = \X_1 = \X_2$, starting from the assumption that there exist $a, b\in X$ with $\phi_1(a)$ comparable with $\phi_2(b)$, we have $c = d = e$ and $m = 1$, whence $\phi_1(c) = \phi_2(c)$.  Hence $\A$ is semi-primal by Theorem~\ref{thm:spcharCorn}.
\end{proof}

It is very easy to derive the following corollary.

\begin{corollary}\label{thm:mainold}
Let $\CB$ be a finite set of finite Cornish algebras of type~$F$ and let $\CY = \{D(\A)\mid \A\in \CB\}$.
Then the algebras in $\CB$ are semi-primal and share a common ternary discriminator term provided $\CY$ has the following properties:
\begin{enumerate}[ \normalfont(i)]
\item
each $\X\in \CY$ has no non-empty proper substructures,
\item
the set $F^-$ is non-empty,

\item
there exists a unary term $t$ in the signature $F$ such that $t^\X$ is constant, for all $\X\in \CY$.
\end{enumerate}
\end{corollary}

\begin{proof}Assume (i)--(iii). If $t\in T^-$, then Theorem~\ref{thm:mainnew} applies. If $t\in T^+$, then $h:=t(g(v))$ belongs to $T^-$ and $h^\X$ is constant, where $g$ is any element of $F^-$, and again Theorem~\ref{thm:mainnew} applies.  
\end{proof}

We can now use our results to show that the characterisation of quasi-primal Ockham algebras (Davey, Nguyen, Pitkethly~\cite{OckAlg}) extends to a characterisation of discriminator varieties of Ockham algebras. We will also use Theorem~\ref{thm:mainnew} and Corollary~\ref{thm:mainold} to give examples of quasi-primal Cornish algebras that share a common ternary discriminator term and to give an infinite family of semi-primal Cornish algebras.

\begin{figure}[ht]
\begin{tikzpicture}
        \node[anchor=west] at (6,0) {($m$ odd)\quad $t(v) := g(v)$};
        \node at (-0.75,0) {$\CT_m$};
     \node[unshaded] (0) at (0,0) {};
         \node[label, anchor=south] at (0) {$0$};
     \node[unshaded] (1) at (1,0) {};
         \node[label, anchor=south] at (1) {$1$};
     \node[unshaded] (2) at (2,0) {};
         \node[label, anchor=south] at (2) {$2$};
     \node[invisible] (3) at (3,0) {};
     \node at (3.5,0) {$\cdots$};
     \node[invisible] (4) at (4,0) {};
     \node[unshaded] (5) at (5,0) {};
         \node[label, anchor=south] at (5) {$m{-}1$};
     \draw[curvy, dotted] (0) to [bend left] (1);
     \draw[curvy, dotted] (1) to [bend left] (2);
     \draw[curvy, dotted] (2) to [bend left] (3);
     \draw[curvy, dotted] (4) to [bend left] (5);
     \draw[curvy, dotted] (5) to [bend left=20] (0);
\end{tikzpicture}
\caption{The Ockham space $\CT_m$}
\label{fig:OckSpace}
\end{figure}

\begin{theorem}[See {\cite[Theorem 4.5]{OckAlg}}]\label{ex:Ock}
A variety of Ockham algebras is a discriminator variety if and only if it is generated by a finite set $\CB$ of Ockham algebras such that, for each $\A\in \CB$, we have $D(\A) \cong \CT_m$, for some odd~$m\in \mathbb N$, where $\CT_m$ is the Ockham space shown in Figure~\ref{fig:OckSpace}: $\langle C_m; \le\rangle$ is an antichain and $g$ is a cycle.
\end{theorem}

\begin{proof}
The backward implication follows from Theorem~\ref{thm:mainnew} with $t(v):= g(v)$. For the forward implication, assume that $\CA$ is a discriminator variety of Ockham algebras. By Theorem~\ref{thm:nodiscvars}, we have $\CA = \Var\CB$ for some finite set $\CB$ of quasi-primal Ockham algebras that share a common ternary discriminator term. Let $\A\in \CB$ and define $\X = D(\A) =\langle X; g, \le, \T\rangle$. It remains to show that $\X \cong \CT_m$, for some odd~$m\in \mathbb N$. 

As $\A$ is quasi-primal and therefore simple, it follows from Lemma~\ref{lem:sub-alg} that $\X$ has no non-empty proper substructures, and consequently $g$ is simply a cycle. Let the size of $\X$ be~$m$. If $m$ is odd, then by Lemma~\ref{lem:antichain}, $\X$ is isomorphic to~$\CT_m$, and we are done. 
Suppose, by way of contradiction, that $m$ is even. We will find an Ockham space $\Y$ and  jointly surjective morphisms $\phi_1, \phi_2\colon \X\to \Y$ such that $\Y_1 :=\phi_1(\X)$ and $\Y_2 :=\phi_2(\X)$ are proper substructures of $\Y$ with $\Y \ne \Y_1\dotcup \Y_2$. By Theorem~\ref{thm:qpcharCorn} this contradicts our assumption that $E(\X)$ is quasi-primal.

Let $\Y = \langle \{u, v, w\}; g, \le, \T\rangle$ be the Ockham space with $u < v < w$ and $g(u) = w$, $g(v) = v$, $g(w) = u$, that is, $\Y$ is the reduct of the Cornish space $\Y_3$ of Figure~\ref{fig:QPCornishfam} obtained by removing the order-preserving map~$f$. We now define the jointly surjective morphisms $\phi_1, \phi_2\colon \X\to \Y$. 
Choose $a$ to be any minimal element of the ordered set $\langle X; \le\rangle$. Then, since the map $g$ is a dual automorphism of $\langle X; \le\rangle$, it follows that $X_1 :=\{g^n(a)\mid n \text{ is even}\}$ is a set of minimal elements of $\langle X; \le\rangle$ and $X_2 :=\{g^n(a)\mid n \text{ is odd}\}$ is a set of maximal elements of $\langle X; \le\rangle$ with $X = X_1\dotcup X_2$. (This is true even if $\langle X; \le\rangle$ happens to be an antichain.) We can now define the required Ockham space morphisms $\phi_1, \phi_2\colon \X \to \Y$ by
$\phi_1(x) = u$, if $x\in X_1$, and $\phi_1(x) = w$, if $x\in X_2$,
and $\phi_2(x) = v$, for all $x\in X$. 
\end{proof}

\begin{figure}[ht]
\begin{tikzpicture}
\begin{scope}
     \node at (0,0) {$\X_1$};
   \node[unshaded] (0) at (0,1) {};
   \node[unshaded] (1) at (1.33,1) {};
   \node[unshaded] (2) at (2.66,1) {};
   \node[unshaded] (3) at (4,1) {};
   \node[unshaded] (5) at (1,0) {};
   \node[unshaded] (6) at (2,0) {};
   \node[unshaded] (7) at (3,0) {};
   \draw[straight, <->] (0) to (5);
   \draw[straight, ->] (3) to (7);
   \draw[loopy] (1) to [out=115,in=75] (1);
   \draw[loopy] (2) to [out=115,in=75] (2);
   \draw[loopy] (6) to [out=115,in=75] (6);
   \draw[loopy] (7) to [out=115,in=75] (7);
   \draw[curvy, dotted] (0) to [bend right=20] (1);
   \draw[curvy, dotted] (1) to [bend right=20] (2);
   \draw[curvy, dotted] (2) to [bend right=20] (3);
   \draw[curvy, dotted] (3) to [bend right] (0);
   \draw[curvy, dotted] (5) to [bend right=20] (6);
   \draw[curvy, dotted] (6) to [bend right=20] (7);
   \draw[curvy, dotted] (7) to [bend right=50, min distance=0.7cm](5);
   \node[anchor=base] at (2,-1) {$t(v) := f^2(g(v))$};
\end{scope}
   \begin{scope}[xshift=5.325cm]
     \node at (0,0) {$\X_2$};
     \node[unshaded] (u) at (1,0) {};
       \node[label, anchor=east] at (u) {$u$};
     \node[unshaded] (v) at (1,1) {};
       \node[label, anchor=east] at (v) {$v$};
     \node[unshaded] (w) at (1,2) {};
       \node[label, anchor=east] at (w) {$w$};
     \draw[order] (u) to (v);
     \draw[order] (u) to (v);
     \draw[order] (v) to (w);
     \draw[curvy, ->] (u) to [bend left] (v);
     \draw[curvy, ->] (v) to [bend left] (w);
     \draw[loopy] (w) to [out=115,in=75] (w);
     \draw[loopy, dotted] (v) to [out=25,in=-15] (v);
     \draw[curvy, <->, dotted] (u) to [bend right=55, min distance=0.5cm] (w);
   \node[anchor=base] at (1,-1) {$t(v) := f^2(g(v))$};
     \end{scope}
\begin{scope}[xshift=8.75cm]
   \node at (0.25,0) {$\X_3$};
   \node[unshaded] (0) at (1,2) {};
   \node[unshaded] (1) at (2,2) {};
   \node[unshaded] (2) at (3,2) {};
   \node[unshaded] (3) at (1,1) {};
   \node[unshaded] (4) at (2,1) {};
   \node[unshaded] (5) at (3,1) {};
   \node[unshaded] (6) at (1,0) {};
   \node[unshaded] (7) at (2,0) {};
   \node[unshaded] (8) at (3,0) {};
   \draw[curvy, dotted] (0) to [bend right=20] (1);
   \draw[curvy, dotted] (1) to [bend right=20] (2);
   \draw[curvy, dotted] (2) to [bend right=35] (0);
   \draw[curvy, dotted] (3) to [bend right=20] (4);
   \draw[curvy, dotted] (4) to [bend right=20] (5);
   \draw[curvy, dotted] (5) to [bend right=35] (3);
   \draw[curvy, dotted] (6) to [bend right=20] (7);
   \draw[curvy, dotted] (7) to [bend right=20] (8);
   \draw[curvy, dotted] (8) to [bend right=35] (6);
   \draw[curvy] (0) to [bend right=20] (3);
   \draw[curvy] (3) to [bend right=20] (6);
   \draw[curvy] (6) to [bend right=35] (0);
   \draw[curvy] (1) to [bend right=20] (4);
   \draw[curvy] (4) to [bend right=20] (7);
   \draw[curvy] (7) to [bend right=35] (1);
   \draw[curvy] (2) to [bend right=20] (5);
   \draw[curvy] (5) to [bend right=20] (8);
   \draw[curvy] (8) to [bend right=35] (2);
   \node[anchor=base] at (2,-1) {$t(v) := g(v)$};
\end{scope}
\end{tikzpicture}
\caption{Some Cornish spaces whose Priestley duals are quasi-primal}
\label{fig:CornSpace}
\end{figure}
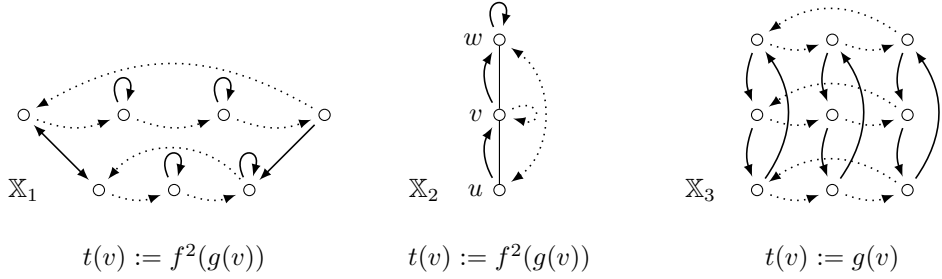
\begin{example}\label{ex:Corn}
Some examples of Cornish spaces $\X$ such that $\A=E(\X)$ is quasi-primal are given in Figure~\ref{fig:CornSpace}. The Cornish spaces are of type $F = \{f, g\}$, where $F^+ = \{f\}$ and $F^- = \{g\}$: the map $f^{\X_i}$ is indicated by solid lines and $g^{\X_i}$ is indicated by dotted lines.
The quasi-primality of $\A_i=E(\X_i)$ follows from Theorem~\ref{thm:mainnew}: in each case, the indicated term $t(v)$ belongs to $T^-$ and gives the required odd cycles. In fact, the term $t(v) :=f^2(g(v))$ belongs to $T^-$ and gives the required odd cycles on $\X_1$, $\X_2$ and $\X_3$. It follows that $\A_1 :=E(\X_1)$, $\A_2 :=E(\X_2)$ and $\A_3 :=E(\X_3)$ are quasi-primal and share a common ternary discriminator term, whence they collectively generate a discriminator variety.
\end{example}

\begin{example}\label{ex:Corn2}
An infinite family $\A_1$, $\A_2$, $\A_3$, \dots\ of semi-primal Cornish algebras of type $F = \{f, g\}$, where $F^+ = \{f\}$ and $F^- = \{g\}$, is given in Figure~\ref{fig:QPCornishfam}. In each case $\A_n := E(\Y_n)$, the maps $f^{\Y_n}$ and $f^{\A_n}$ are indicated by solid lines, and $g^{\Y_n}$ and $g^{\A_n}$ are indicated by dotted lines. The semi-primality of $\A_n$ follows from Corollary~\ref{thm:mainold}: the operation $g$ belongs to $F^-$ and $t(v) = f^{n-1}(v)$ is one possible choice for the term required in (iii).
\end{example}


\begin{figure}[ht]
\begin{tikzpicture}
   \begin{scope}
     \node at (1,-1) {$\Y_2$};
     \node[unshaded] (u) at (1,0) {};
       \node[label, anchor=east] at (u) {$u_1$};
     \node[unshaded] (v) at (1,1) {};
       \node[label, anchor=east] at (v) {$u_2$};
     \draw[order] (u) to (v);
     \draw[curvy] (u) to [bend left] (v);
     \draw[loopy] (v) to [out=115,in=75] (v);
     \draw[curvy, <->, dotted] (u) to [bend right] (v);
     \end{scope}
   \begin{scope}[xshift=1.75cm]
     \node at (1,-1) {$\A_2$};
     \node[unshaded] (0) at (1,0) {};
       \node[label, anchor=east] at (0) {$0$};
     \node[unshaded] (1) at (1,1) {};
       \node[label, anchor=east] at (1) {$a_1$};
     \node[unshaded] (2) at (1,2) {};
       \node[label, anchor=east] at (2) {$1$};
     \draw[order] (0) to (1);
     \draw[order] (1) to (2);
     \draw[loopy] (2) to [out=115,in=75] (2);
     \draw[loopy] (0) to [out=-65,in=-105] (0);
     \draw[loopy, dotted] (1) to [out=25,in=-15] (1);
     \draw[curvy] (1) to [bend left] (2);
     \draw[curvy, <->, dotted] (0) to [bend right=55, min distance=0.5cm] (2);
     \end{scope}
   \begin{scope}[xshift=4.25cm]
     \node at (1,-1) {$\Y_3$};
     \node[unshaded] (u) at (1,0) {};
       \node[label, anchor=east] at (u) {$u_1$};
     \node[unshaded] (v) at (1,1) {};
       \node[label, anchor=east] at (v) {$u_2$};
     \node[unshaded] (w) at (1,2) {};
       \node[label, anchor=east] at (w) {$u_3$};
     \draw[order] (u) to (v);
     \draw[order] (v) to (w);
     \draw[curvy, ->] (u) to [bend left] (v);
     \draw[curvy, ->] (v) to [bend left] (w);
     \draw[loopy] (w) to [out=115,in=75] (w);
     \draw[loopy, dotted] (v) to [out=25,in=-15] (v);
     \draw[curvy, <->, dotted] (u) to [bend right=55, min distance=0.5cm] (w);
     \end{scope}
   \begin{scope}[xshift=6cm]
     \node at (1,-1) {$\A_3$};
     \node[unshaded] (0) at (1,0) {};
       \node[label, anchor=east] at (0) {$0$};
     \node[unshaded] (1) at (1,1) {};
       \node[label, anchor=east] at (1) {$a_1$};
     \node[unshaded] (2) at (1,2) {};
       \node[label, anchor=east] at (2) {$a_2$};
     \node[unshaded] (3) at (1,3) {};
       \node[label, anchor=east] at (3) {$1$};
     \draw[order] (0) to (1);
     \draw[order] (1) to (2);
     \draw[order] (2) to (3);
     \draw[loopy] (3) to [out=115,in=75] (3);
     \draw[loopy] (0) to [out=-65,in=-105] (0);
     \draw[curvy, <->, dotted] (1) to [bend right] (2);
     \draw[curvy, <->, dotted] (0) to [bend right] (3);
     \draw[curvy] (1) to [bend left] (2);
     \draw[curvy] (2) to [bend left] (3);
     \end{scope}
   \begin{scope}[xshift=8.5cm]
     \node at (1,-1) {$\Y_4$};
     \node[unshaded] (u) at (1,0) {};
       \node[label, anchor=east] at (u) {$u_1$};
     \node[unshaded] (v) at (1,1) {};
       \node[label, anchor=east] at (v) {$u_2$};
     \node[unshaded] (w) at (1,2) {};
       \node[label, anchor=east] at (w) {$u_3$};
     \node[unshaded] (x) at (1,3) {};
       \node[label, anchor=east] at (x) {$u_4$};
     \draw[order] (u) to (v);
     \draw[order] (v) to (w);
     \draw[order] (w) to (x);
     \draw[curvy, ->] (u) to [bend left] (v);
     \draw[curvy, ->] (v) to [bend left] (w);
     \draw[curvy, ->] (w) to [bend left] (x);
     \draw[loopy] (x) to [out=115,in=75] (x);
     \draw[curvy, <->, dotted] (u) to [bend right] (x);
     \draw[curvy, <->, dotted] (v) to [bend right] (w);
     \end{scope}
   \begin{scope}[xshift=10.25cm]
     \node at (1,-1) {$\A_4$};
     \node[unshaded] (0) at (1,0) {};
       \node[label, anchor=east] at (0) {$0$};
     \node[unshaded] (1) at (1,1) {};
       \node[label, anchor=east] at (1) {$a_1$};
     \node[unshaded] (2) at (1,2) {};
       \node[label, anchor=east] at (2) {$a_2$};
     \node[unshaded] (3) at (1,3) {};
       \node[label, anchor=east] at (3) {$a_3$};
     \node[unshaded] (4) at (1,4) {};
       \node[label, anchor=east] at (4) {$1$};
     \draw[order] (0) to (1);
     \draw[order] (1) to (2);
     \draw[order] (2) to (3);
     \draw[order] (3) to (4);
     \draw[loopy] (4) to [out=115,in=75] (4);
     \draw[loopy] (0) to [out=-65,in=-105] (0);
     \draw[loopy, dotted] (2) to [out=25,in=-15] (2);
     \draw[curvy, <->, dotted] (1) to [bend right=55, min distance=0.5cm] (3);
     \draw[curvy, <->, dotted] (0) to [bend right=40, min distance=0.1cm] (4);
     \draw[curvy] (1) to [bend left] (2);
     \draw[curvy] (2) to [bend left] (3);
     \draw[curvy] (3) to [bend left] (4);
     \end{scope}
\end{tikzpicture}
\caption{An infinite family of semi-primal Cornish algebras.}\label{fig:QPCornishfam}
\end{figure}
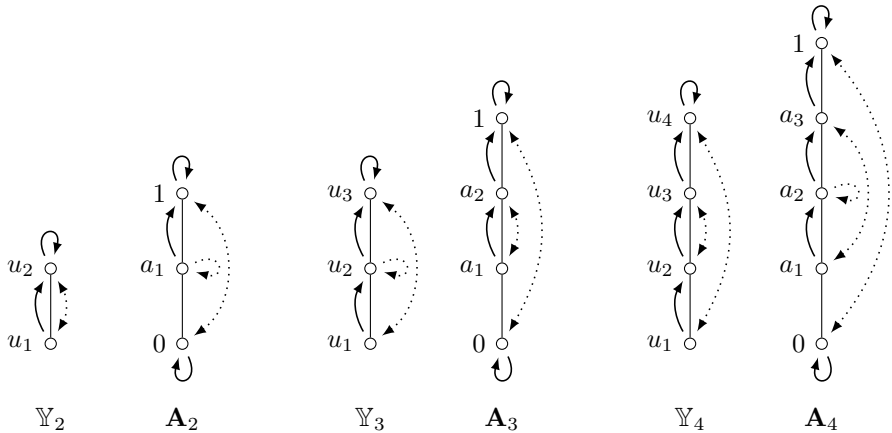


\begin{thebibliography}{99}

\bibitem{JoyOfCats}
Ad\'amek, J., Herrlich, H., Strecker, G.\,E.: Abstract and Concrete Categories: the Joy of Cats. Available at \url{http://katmat.math.uni-bremen.de/acc}

 \bibitem{BP:nu}
Baker, K.\,A., Pixley, A.\,F.: Polynomial interpolation and the
Chinese remainder theorem for algebraic systems. Math. Z. \textbf{143}, 165--174 (1975)

\bibitem{BS}
Burris, S., Sankappanavar, H.\,P.: A course in universal algebra. The Millennium Edition, 2012 Update. Available from \url{https://www.math.uwaterloo.ca/~snburris/htdocs/ualg.html}

\bibitem{CD98}
Clark, D.\,M., Davey, B.\,A.: Natural Dualities for the Working Algebraist. Cambridge University Press, New York (1998)

\bibitem{Corn77}
Cornish, W.\,H.: Monoids acting on distributive lattices, manuscript (invited talk at the annual meeting of the Austral. Math. Soc., May 1977)

\bibitem{Corn86}
Cornish, W.\,H.: Antimorphic Action: Categories of Algebraic Structures with Involutions or Anti-endomorphisms. Research and Exposition in Mathematics, vol. 12, Heldermann, Berlin (1986)

\bibitem{CF77}
Cornish, W.\,H., Fowler, P.\,R.: Coproducts of De Morgan algebras. Bull. Aust. Math. Soc. \textbf{16}, 1--13 (1977)

\bibitem{CF79}
Cornish, W.\,H., Fowler, P.\,R.: Coproducts of Kleene algebras. J. Austral. Math. Soc. (Ser. A) \textbf{27} 209--220 (1979)

\bibitem{BD78}
Davey, B.\,A.: Subdirectly irreducible distributive double p-algebras. Algebra Universalis \textbf{8}, 73--88 (1978)

\bibitem{DG02}
Davey, B.\,A., Galati, J.\,C.: A coalgebraic view of Heyting duality, Studia Logica \textbf{75}, 259--270 (2003)

\bibitem{DH04}
Davey, B.\,A., Haviar, M.: Applications of Priestley duality in transferring optimal dualities. Studia Logica \textbf{78}, 213--236 (2004)

\bibitem{OckAlg}
Davey, B.\,A., Nguyen, L.\,T., Pitkethly, J.\,G.: Counting relations on Ockham algebras. Algebra Universalis \textbf{74}, 35--63 (2015) 

\bibitem{DP87}
Davey, B.\,A., Priestley, H.\,A.: Generalised piggyback dualities and applications to Ockham algebras. Houston J. Math. \textbf{13}, 151--198 (1987)

\bibitem{DP93}
Davey, B.\,A., Priestley, H.\,A.: Optimal natural dualities. Trans. Amer. Math. Soc. \textbf{338}, 655--677 (1993)

\bibitem{DP96}
Davey, B.\,A., Priestley, H.\,A.: Optimal natural dualities for varieties of Heyting algebras. Studia Logica \textbf{56}, 67--96 (1996)

\bibitem{DSW:91}
Davey, B.\,A., Schumann, V.\,J., Werner, H.: From the subalgebras of the square to the discriminator. Algebra Universalis 28, 500--519 (1991)

\bibitem{ILO}
Davey, B.\,A., Priestley, H.\,A.: Introduction to Lattices and Order, 2nd edn. Cambridge University Press, New York (2002) 

\bibitem{Esa74}
Esakia, L.\,L.: Topological Kripke models. Soviet Math. Dokl. \textbf{15}, 147--151 (1974)

\bibitem{FP64}
Foster, A.\,L., Pixley, A.: Semi-categorical algebras. I. Semi-primal algebras. Math. Z. \textbf{85}, 147--169 (1964)

\bibitem{Kat73}
Katri{\v{n}}{\'a}k, T.: The structure of distributive double p-algebras. Regularity and congruences. Algebra Universalis \textbf{3}, 238--246 (1973)

\bibitem{Mart90}
Mart\'inez, N.\,G.: The Priestley duality for Wajsberg algebras. Studia Logica \textbf{49}, 31--46 (1990)

\bibitem{MartPri98}
Mart\'inez, N.\,G., Priestley, H.\,A.: On Priestley spaces of lattice-ordered algebraic structures. Order \textbf{15}, 297--323 (1998)

\bibitem{McKV89}
McKenzie, R., Valeriote, M.: The structure of decidable locally finite varieties. Progress in Mathematics, vol. 79. Birkh\"auser, Boston (1989)

\bibitem{P70}
Pixley, A.\,F.: Functionally complete algebras generating distributive and permutable classes. Math. Z. \textbf{114}, 361--372 (1970)

\bibitem{P71}
Pixley, A.\,F.: The ternary discriminator function in universal algebra. Math. Ann. \textbf{191}, 167--180 (1971)

\bibitem{Pri70}
Priestley, H.\,A.: Representation of distributive lattices by means of ordered Stone spaces. Bull. London Math. Soc. \textbf{2}, 186--190 (1970)

\bibitem{Pri72}
Priestley, H.\,A.: Ordered topological spaces and the representation of distributive lattices. Proc. London Math. Soc. (3) \textbf{24}, 507--530 (1972)

\bibitem{Pri74}
Priestley, H.\,A.: Stone lattices: a topological approach. Fund. Math. \textbf{84}, 127--143 (1974) 

\bibitem{Pri75}
Priestley, H.\,A.: The construction of spaces dual to pseudocomplemented distributive lattices. Quart. J. Math. Oxford Ser. (2) \textbf{26}, 215--228 (1975)

\bibitem{Pri84}
Priestley, H.\,A.: Ordered sets and duality for distributive lattices. In: Pouzet, M., Richard, D. (eds.) Orders: description and roles (L'Arbresle, 1982). North-Holland Math. Stud., vol. 99, pp. 39--60. North-Holland, Amsterdam (1984) 
 
\bibitem{Pri97}
Priestley, H.\,A.: Varieties of distributive lattices with unary operations. I. J. Austral. Math. Soc. Ser. A \textbf{63}, 165--207 (1997)
 
\bibitem{PriSan98}
Priestley, H.\,A., Santos, R.: Varieties of distributive lattices with unary operations. II. Portugal. Math. \textbf{55}, 135--166 (1998)

\bibitem{San85}
Sankappanavar, H.\,P.: Heyting algebras with dual pseudocomplementation. Pacific J. Math. \textbf{117}, 405--415 (1985)

\bibitem{Tay16} 
Taylor, C.\,J.: Discriminator varieties of double Heyting algebras. Rep. Math. logic (in press)

\bibitem{Urq79}
Urquhart, A.: Distributive lattices with a dual homomorphic operation. Studia Logica \textbf{38}, 201--209 (1979)

\bibitem{W70}
Werner, H.: Eine Charakterisierung funktional vollst{\"a}ndiger Algebren. Arch. Math. (Basel) \textbf{21}, 381--385 (1970)

\bibitem{W78}
Werner, H.: Discriminator-algebras. Algebraic representation and model theoretic properties. Studien zur Algebra und ihre Anwendungen, Band 6, Akademie-Verlag, Berlin (1978)

\end{thebibliography}
\end{document}